\documentclass[a4paper]{amsart}

\usepackage[T2A]{fontenc}
\usepackage[utf8]{inputenc}
\usepackage{geometry}
\usepackage{graphicx}

\usepackage{booktabs}
\usepackage{siunitx}

\usepackage{amsmath,amssymb,amsthm,mathtools}
\usepackage{hyperref}
\usepackage{url}
\usepackage{tikz}

\newtheorem{theorem}{Theorem}
\newtheorem{lemma}{Lemma}
\newtheorem{corollary}{Corollary}

\theoremstyle{definition}
\newtheorem{definition}{Definition}

\newtheorem{example}{Example}
\newtheorem*{convention}{Convention}

\theoremstyle{remark}
\newtheorem{remark}{Remark}

\newcommand{\R}{\mathbb {R}}
\newcommand{\Z}{\mathbb {Z}}
\newcommand{\F}{\mathbb {F}}

\newcommand{\oeislink}[1]{\href{https://oeis.org/#1}{\textcolor{blue}{\underline{#1}}}}

\author{Yury Belousov}
\address{Yury Belousov \\ Leonhard Euler International Mathematical Institute in Saint Petersburg} 
\email{bus99@yandex.ru}

\author{Viktoriia Georgievskaia}
\address{Viktoriia Georgievskaia \\ Saint Petersburg State University} 
\email{vika-g-2005@mail.ru}

\thanks{This work was funded by the Russian Science Foundation, grant No. 25-11-00251, \href{https://rscf.ru/project/25-11-00251/}{https://rscf.ru/project/25-11-00251/}.}

\title{S-meandric Permutations and Tangency Polynomials}

\begin{document}
\begin{abstract}
     A meander is a configuration of two simple plane curves intersecting transversely. The orders of their intersection points define a permutation that determines the configuration. When tangencies are allowed, however, different configurations can share the same permutation. We study the combinatorial and algebraic structures arising from this non-uniqueness. We give a realization criterion and show that the realizations of each realizable permutation form an affine space over the two-element field. We describe this space using an associated graph, called the component spine. We prove that the component spine of every permutation is a cactus. We also introduce the tangency polynomial, which counts realizations by their number of tangencies, investigate its properties, and prove that it factors over the cycles and bridges of the component spine. We derive a central limit theorem for tangency counts and obtain asymptotic formulas for the number of distinct tangency polynomials.
\end{abstract}
\maketitle

\section{Introduction}
A meander is a configuration of a pair of simple plane curves. 
In the non-singular case, all intersection points of the two curves are transverse (see examples in Fig.~\ref{fig:meander}). The problem of counting such configurations was formulated by V. Arnol’d~\cite{Arnold88BranchedCovering}; a closely related problem for closed meanders, under the name of planar permutations, was considered earlier by P. Rosenstiehl~\cite{Rosenstiehl84PlanarPermutations}. Meanders have subsequently appeared in enumerative combinatorics, the theory of noncrossing partitions, statistical physics, low-dimensional topology, and the geometry of moduli spaces. We refer the reader to the survey~\cite{Zvonkin23MeandersPersonalPerspective} for historical background and an overview of these connections. Despite considerable progress, the general enumeration problem and the precise asymptotic behavior of meander numbers remain open.

\begin{figure}[h]
    \centering
\begin{tikzpicture}[scale=5.5,
    intersection label/.style={font=\scriptsize, inner sep=0.5pt}
]

\draw[thick] (0,0) -- (1,0);
\draw[ultra thick, ->] (0.0158771,0.125)
    to[out=0, in=90, distance=10.9956] (0.875,0)
    to[out=-90, in=-90, distance=1.5708] (0.75,0)
    to[out=90, in=90, distance=1.5708] (0.625,0)
    to[out=-90, in=-90, distance=1.5708] (0.5,0)
    to[out=90, in=90, distance=1.5708] (0.375,0)
    to[out=-90, in=-90, distance=1.5708] (0.25,0)
    to[out=90, in=90, distance=1.5708] (0.125,0)
    to[out=-90, in=180, distance=10.9956] (0.984123,-0.125);

\foreach \x [count=\i] in
    {0.106,0.27,0.355,0.52,0.604,0.77,0.9}
    \node[intersection label, above=2pt] at (\x,0) {\i};

\node at (0.5,-0.35) {$(7,6,5,4,3,2,1)$};
\end{tikzpicture}
\hspace{2cm}
\begin{tikzpicture}[scale=5.5,
    intersection label/.style={font=\scriptsize, inner sep=0.5pt}
]

\draw[thick] (0,0) -- (1,0);
\draw[ultra thick, ->] (0.0355579,0.185185)
    to[out=0, in=90, distance=6.98132] (0.555556,0)
    to[out=-90, in=-90, distance=1.39626] (0.444444,0)
    to[out=90, in=90, distance=4.18879] (0.111111,0)
    to[out=-90, in=-90, distance=9.77384] (0.888889,0)
    to[out=90, in=90, distance=1.39626] (0.777778,0)
    to[out=-90, in=-90, distance=6.98132] (0.222222,0)
    to[out=90, in=90, distance=1.39626] (0.333333,0)
    to[out=-90, in=-90, distance=4.18879] (0.666667,0)
    to[out=90, in=180, distance=4.18879] (0.964442,0.185185);

\foreach \x [count=\i] in
    {0.08,0.20,0.355,0.47,
     0.58,0.64,0.755,0.915}
    \node[intersection label, above=1.5pt] at (\x,0) {\i};

\node at (0.5,-0.35) {$(5,4,1,8,7,2,3,6)$};
\end{tikzpicture}
    \caption{Meanders and their permutations.}
    \label{fig:meander}
\end{figure}
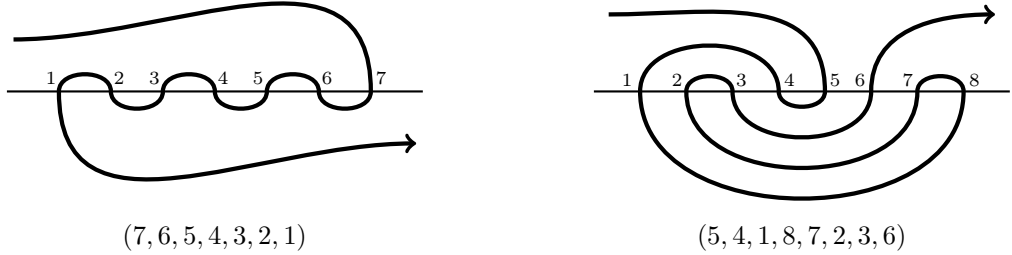

It is convenient to encode a meander by a permutation. The idea is already illustrated in Fig.~\ref{fig:meander}, and for a formal definition, see Definition~\ref{def:meander permutation}. The permutation uniquely determines meanders, and the question of whether a given permutation is meandric (that is, whether it arises as the permutation of some meander) has been solved using several different methods in~\cite{Rosenstiehl84PlanarPermutations, RosenstiehlTarjan84GaussCodes, HoffmannMehlhornRosenstiehlTarjan86JordanSequences, ShtyllaTraldiZulli09Realization, Lopatkin2022Meandric}.

The situation changes when non-transverse intersections are allowed. Such meanders are called singular. Singular meanders arose in the study of meander factorization: the insertion and decomposition operations are most naturally expressed in a class allowing both transverse intersections and tangencies (see~\cite{Belousov26PrimeFactorization, Belousov25SingularMeanders}). For singular meanders, the permutation itself is no longer a complete invariant, and distinct singular meanders can share the same permutation; see an example in Fig.~\ref{fig:meanders with same permutation}. The aim of this paper is to study this non-uniqueness and the additional structure carried by the set of realizations of a permutation.

We call a permutation \emph{s-meandric} if it is realized by some singular meander. To study the tangencies in its realizations, we introduce the \emph{tangency polynomial} $T_\pi(x)$: its coefficient of $x^k$ is the number of realizations with exactly $k$ non-transverse intersections. For example, the four realizations in Fig.~\ref{fig:meanders with same permutation} give $T_{(3,2,1)}(x)=1+3x^2$. The main areas of our research are the following.

\subsection*{Structure of the realizations}
In Section~\ref{sec:realization criteria} we describe the set of realizations of a permutation using
its interlacement graph. We prove that a permutation is s-meandric if and only if this graph is bipartite (Theorem~\ref{thm:realization criteria}), and that its realizations correspond bijectively to normalized proper $2$-colorings (Theorem~\ref{thm:degree of realizability}).

In Section~\ref{sec:component spine} we describe how the choice of a proper coloring affects the intersection types. We introduce the \emph{component spine}, which records how the connected components of the interlacement graph occur along the permutation. Our main structural result is that the component spine of every permutation is a cactus: each edge belongs to at most one simple cycle (Theorem~\ref{thm:cactus}). For an s-meandric permutation, we also use this graph to describe the natural affine-space structure over $\mathbb{F}_2$ carried by its set of realizations (Theorem~\ref{thm:cut space}).

\subsection*{Tangency polynomials}
In Section~\ref{sec:tangency polynomial} we study tangency polynomials. The main result of the section is the block factorization theorem (Theorem~\ref{thm:block factorization}): it allows us to study tangencies separately on the bridges and cycle blocks of the component spine. In the subsequent subsections we establish this factorization and study its consequences. Among all of them we mentioned here the log-concavity/unimodality of the tangency polynomial coefficients (Theorem~\ref{thm:polynomial coefficients}) and the central limit theorem (Theorem~\ref{thm:central limit theorem}).

\subsection*{Enumeration and asymptotics}
In Section~\ref{sec:asymptotics and enumeration} we consider the problem of counting s-meandric permutations of a given order. We prove that these numbers have a well-defined exponential growth rate and give upper and lower bounds for it (Theorem~\ref{thm:estimate on s-meandric number}). We also compute the numbers up to order~$24$.

In Section~\ref{sec:number-tangency-polynomials} we consider the problem of counting distinct tangency polynomials (as different collections of block factors may give the same polynomial). We show that suitable collections of these factors can be realized (Theorem~\ref{thm:visible realization}). Then we describe all multiplicative identities between block factors (Lemma~\ref{lem:block factors identity}) and obtain a unique normal form (Theorem~\ref{thm:polynomial normal form}). Together, these results allow us to obtain precise asymptotic for the number of distinct tangency polynomials (Theorems~\ref{thm:tangecy polinomial asymptotic} and \ref{thm:ordinary tangecy polinomial asymptotic}). More precisely, let $N_n$ be the number of distinct nonzero tangency polynomials of permutations in $S_n$, and let $O_n$ be the number of those arising from meandric permutations. Then
$$
    N_n \sim \frac{\exp(\pi\sqrt{n})\log n}{8\pi^2},
    \qquad
    O_n \sim
    \frac{\exp\bigl(\pi\sqrt{2n/3}\bigr)}{2\pi\sqrt{2n}}
$$
as $n\to\infty$. Enumeration data are collected in Appendix~\ref{appendix}.

\subsection*{Acknowledgments}
    The authors would like to thank Andrey Malyutin and Vadim Stepaniuk for many fruitful discussions. 

\section{Realization criteria}\label{sec:realization criteria}
Since we are primarily interested in the combinatorial aspects of meanders, we use a normalized planar model in which the intersection points are fixed and the arcs between successive intersections are semicircles. Each meander is then represented canonically by finite data. Because irrelevant topological choices have been removed, no equivalence relation is needed, and the passage from meanders to their permutations becomes direct. Compare this with the more topological treatment presented in~\cite{Belousov26PrimeFactorization}.

\begin{definition}\label{def:singular meanders}
Let $\ell=\R\times\{0\}$, $H^\pm=\left\{(x,y) \in \R^2 \mid \pm y>0\right\}$, and for $k \in \Z$, put $p_k=(k,0)\in \R^2$. A \emph{singular meander of order $n \geq 1$} is a simple arc $M \subset \R^2$ from $p_0$ to $p_{n+1}$ such that 
\begin{itemize}
    \item $M\cap\ell=\{p_0,p_1,\ldots,p_{n+1}\}$;
    \item the closure of every component of $M\setminus\ell$ is a semicircle with diameter contained in $\ell$;
    \item the semicircle incident to $p_0$ lies in $H^+$.
\end{itemize}

For $1\leq k\leq n$, an intersection point $p_k$ is called a \emph{touch} if the two semicircles incident to it lie in the same half-plane, either $H^+$ or $H^-$. We denote the number of touches of $M$ by $\operatorname{t}(M)$. A singular meander $M$ is said to be \emph{non-singular} if $\operatorname{t}(M)=0$.
\end{definition}

\begin{definition}\label{def:meander permutation}
    Let $M$ be a singular meander of order $n$. A permutation $\pi \in S_n$ is called \emph{associated with $M$} if, when traversing $M$ from $p_0$ to $p_{n+1}$, we encounter its intersection points in the order $p_{\pi(1)},p_{\pi(2)},\dots,p_{\pi(n)}$. 
    Equivalently, we say that $M$ \emph{realizes} $\pi$.
\end{definition}

We write permutations in one-line notation, $\pi=(\pi(1),\dots,\pi(n))$.

\begin{remark}
    Note that a permutation may be realized by several singular meanders, but by at most one non-singular meander. See Fig.~\ref{fig:meanders with same permutation} for an example.
\end{remark}

\begin{figure}[ht]
    \centering
    \begin{tikzpicture}[scale = 2.5]
    \draw[thick] (-0.05, 0) to (1.05, 0);
    \draw[ultra thick] (0.0, 0.0)
    	to[out = 90, in = 90, distance = 9.42478] (0.75, 0)
    	to[out = -90, in = -90, distance = 3.14159] (0.5, 0)
    	to[out = 90, in = 90, distance = 3.14159] (0.25, 0)
    	to[out = -90, in = -90, distance = 9.42478] (1, 0);
    \end{tikzpicture}
    \hspace{0.4cm}
    \begin{tikzpicture}[scale = 2.5]
    \draw[thick] (-0.05, 0) to (1.05, 0);
    \draw[ultra thick] (0, 0)
    	to[out = 90, in = 90, distance = 9.42478] (0.75, 0)
    	to[out = 90, in = 90, distance = 3.14159] (0.5, 0)
    	to[out = 90, in = 90, distance = 3.14159] (0.25, 0)
    	to[out = -90, in = -90, distance = 9.42478] (1, 0);
    \end{tikzpicture}
    \hspace{0.4cm}
    \begin{tikzpicture}[scale = 2.5]
    \draw[thick] (-0.05, 0) to (1.05, 0);
    \draw[ultra thick] (0, 0)
    	to[out = 90, in = 90, distance = 9.42478] (0.75, 0)
    	to[out = 90, in = 90, distance = 3.14159] (0.5, 0)
    	to[out = -90, in = -90, distance = 3.14159] (0.25, 0)
    	to[out = -90, in = -90, distance = 9.42478] (1, 0);
    \end{tikzpicture}
    \hspace{0.4cm}
    \begin{tikzpicture}[scale = 2.5]
    \draw[thick] (-0.05, 0) to (1.05, 0);
    \draw[ultra thick] (0, 0)
    	to[out = 90, in = 90, distance = 9.42478] (0.75, 0)
    	to[out = -90, in = -90, distance = 3.14159] (0.5, 0)
    	to[out = -90, in = -90, distance = 3.14159] (0.25, 0)
    	to[out = -90, in = -90, distance = 9.42478] (1, 0);
    \end{tikzpicture}
    \caption{All singular meanders realizing the permutation $(3,2,1)$.}
    \label{fig:meanders with same permutation}
\end{figure}
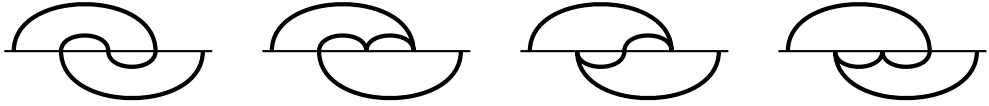

\begin{definition}\label{def:s-meandric permutation}
    A permutation $\pi\in S_n$ is called \emph{s-meandric} if it is realized by some singular meander. It is called \emph{meandric} if it is realized by some non-singular meander.
\end{definition}

\begin{example}
    The permutation $(2,4,1,3)$ is not s-meandric, $(2,1,3,4)$ is s-meandric but not meandric, and $(3,2,1,4)$ is meandric.
\end{example}

\begin{definition} \label{def:interlacement graph}
    Let $\pi\in S_n$ be a permutation. Extend it by setting $\pi(0)=0$ and $\pi(n+1)=n+1$. For each $i\in\{0,\dots,n\}$, the pair $A_i=(\pi(i),\pi(i+1))$ is called an \emph{arc} of $\pi$. Two arcs $(a_1,a_2)$ and $(b_1,b_2)$ are said to \emph{interlace} if either
    $$\min(a_1,a_2)<\min(b_1,b_2)<\max(a_1,a_2)<\max(b_1,b_2)$$
    or
    $$\min(b_1,b_2)<\min(a_1,a_2)<\max(b_1,b_2)<\max(a_1,a_2).$$
    The \emph{interlacement graph} of $\pi$ is the graph whose vertices are $A_0,\dots,A_n$ and in which two vertices are joined by an edge if and only if the corresponding arcs interlace.
\end{definition}

\begin{convention}
    Throughout this paper, we denote the interlacement graph of a permutation $\pi$ by $G_\pi$.
\end{convention}

\begin{theorem}\label{thm:realization criteria}
    A permutation is s-meandric if and only if its interlacement graph is bipartite.
\end{theorem}

\begin{proof}
    Let $\pi\in S_n$. Suppose that $M$ realizes $\pi$. For $0\leq i\leq n$, let $s_i$ be the semicircle of $M$ with endpoints $p_{\pi(i)}$ and $p_{\pi(i+1)}$. Color $A_i$ by $1$ or $0$ according to whether $s_i\setminus\ell$ lies in $H^+$ or $H^-$. Two distinct semicircles in the same half-plane meet away from $\ell$ if and only if their endpoints strictly alternate. Since $M$ is simple, interlacing arcs receive different colors. Thus this coloring is proper, and $G_\pi$ is bipartite.

    Conversely, suppose that $G_\pi$ is bipartite, and choose a proper $2$-coloring with $A_0$ colored $1$. For each $i\in\{0,\dots,n\}$, draw the semicircle $s_i$ with endpoints $p_{\pi(i)}$ and $p_{\pi(i+1)}$ in the half-plane prescribed by the color of $A_i$. The same observation shows that these semicircles meet only at the shared endpoints of consecutive ones. Their union is therefore a simple arc from $p_0$ to $p_{n+1}$, passing through $p_{\pi(1)},\dots,p_{\pi(n)}$ in this order. The choice of the color of $A_0$ ensures the normalization at $p_0$, so this arc is a singular meander realizing $\pi$.
\end{proof}

\begin{remark}
    The analogous criterion for closed singular meanders appears in the paper of Panayotopoulos~\cite[Section~3, p.~18]{Panayotopoulos89GeneralizedPlanarPermutations}, where the associated circular permutations are called \emph{generalized planar permutations}.
\end{remark}

\begin{definition}\label{def:degree of realizability}
    Let $\pi\in S_n$. The \emph{degree of realizability} of $\pi$, denoted by $d(\pi)$, is the number of singular meanders that realize $\pi$.
\end{definition}

\begin{theorem}\label{thm:degree of realizability}
    Singular meanders realizing a permutation $\pi$ are in bijection with proper $2$-colorings of $G_\pi$ in which the arc $A_0$ of $\pi$ is colored $1$. Consequently, the degree of realizability of an s-meandric permutation $\pi$ is given by
    $$
    d(\pi)=2^{q-1},
    $$
    where $q$ is the number of connected components of $G_\pi$.
\end{theorem}

\begin{proof}
    The construction in the proof of Theorem~\ref{thm:realization criteria} associates a realization of $\pi$ with every proper $2$-coloring of $G_\pi$ in which $A_0$ is colored $1$. Conversely, a realization determines this coloring by recording the half-plane of each semicircle. Since a semicircle is uniquely determined by its endpoints and half-plane, these constructions are inverse, proving the bijection.

    If $\pi$ is s-meandric, each of the $q$ connected components of $G_\pi$ is bipartite and has exactly two proper $2$-colorings, exchanged by reversing all its colors. The condition on $A_0$ fixes the coloring of its component, while the remaining $q-1$ components can be colored independently. Hence $d(\pi)=2^{q-1}$.
\end{proof}

\section{Component spine and cactus theorem}\label{sec:component spine}
\begin{definition}\label{def:component spine}
    Let $\pi \in S_n$ be a permutation with arcs $A_0,A_1,\dots,A_n$, and, for each $k \in \{0,1,\dots,n\}$, let $c_k$ be the connected component of $G_\pi$ containing $A_k$. The \emph{component spine} of $\pi$ is the multigraph whose vertices are the distinct components among $c_0,c_1,\dots,c_n$ and which has an edge joining $c_i$ and $c_{i+1}$ for each $i \in \{0,1,\dots,n-1\}$. Loops and parallel edges are retained.
\end{definition}

\begin{remark}
    Equivalently, the component spine is the quotient multigraph of the path $A_0 A_1\cdots A_n$ obtained by identifying two vertices if the corresponding arcs belong to the same connected component of $G_\pi$.
\end{remark}

\begin{convention}
     Throughout this paper, we denote the component spine of a permutation $\pi$ by $K_\pi$.
\end{convention}

\begin{remark}\label{rem:component spine has Euler trail}
    By definition, the edges, taken in their defining order, form an Euler trail with vertex sequence $c_0,c_1,\dots,c_n$. Therefore, the component spine is connected and has either zero or two odd-degree vertices.
\end{remark}

\begin{example}\label{ex:component spine}
    Let $\pi=(3,2,1,6,5,4,9,8,7,11,12,10)$, and let $A_0,A_1,\dots,A_{12}$ be its arcs. The interlacement graph $G_\pi$ has 5 edges and 8 connected components. Thus, the component spine $K_\pi$ has 8 vertices. In Fig.~\ref{fig:component spine}, the unique component that consists of more than one arc is denoted by $C$, and all other components are denoted by the labels of their corresponding arcs.
    
    \begin{figure}[h]
    \centering
    \resizebox{0.9\linewidth}{!}{
    \begin{tikzpicture}
        \foreach \k in {0,...,12}{
            \node[circle,draw,minimum size=7mm,inner sep=0pt]
            (A\k) at ({90-\k*360/13}:2.5) {$A_{\k}$};
        }
        \draw (A0)  to[bend right=20] (A3);
        \draw (A3)  to[bend right=15] (A6);
        \draw (A6)  to[bend right=15] (A9);
        \draw (A9)  to[bend right=20]  (A11);
        \draw (A9)  to[bend right=24]  (A12);
        
        \node at (0,-3.5) {$G_\pi$};
    \end{tikzpicture}
    \hspace{1.5cm}
    \begin{tikzpicture}[vertex/.style={circle,draw,minimum size=7mm,inner sep=0pt}]
        \node[vertex] (X) at (0,0) {$C$};
        
        \node[vertex] (A1)  at (50:2.5)   {$A_1$};
        \node[vertex] (A2)  at (-5:2.5)    {$A_2$};
        \node[vertex] (A4)  at (-45:2.5)  {$A_4$};
        \node[vertex] (A5)  at (-95:2.5)  {$A_5$};
        \node[vertex] (A7)  at (-135:2.5) {$A_7$};
        \node[vertex] (A8)  at (180:2.5)  {$A_8$};
        \node[vertex] (A10) at (145:2.5)  {$A_{10}$};
        
        \draw (X) -- (A1) -- (A2) -- (X);
        \draw (X) -- (A4) -- (A5) -- (X);
        \draw (X) -- (A7) -- (A8) -- (X);
        
        \draw (X) to[bend left=20] (A10);
        \draw (X) to[bend right=20] (A10);
        
        \draw (X) to[loop above,min distance=25mm] (X);
        \node at (0,-3.5) {$K_\pi$};
    \end{tikzpicture}
    }
    \caption{The interlacement graph and the component spine of the permutation $(3,2,1,6,5,4,9,8,7,11,12,10)$.}
    \label{fig:component spine}
    \end{figure}
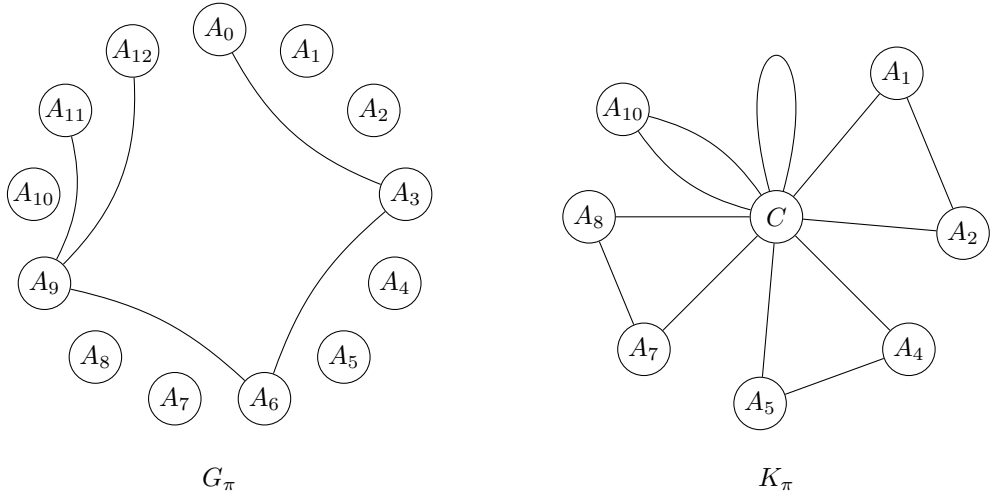
\end{example}

\begin{definition}
    A connected multigraph is called a \emph{cactus} if every edge belongs to at most one simple cycle (loops and $2$-cycles are also regarded as simple cycles).
\end{definition}

\begin{theorem}\label{thm:cactus}
    Let $\pi\in S_n$. Then its component spine is a cactus.
\end{theorem}
\begin{remark}
    The theorem does not require the permutation to be s-meandric.
\end{remark}
\begin{proof}
    Consider a permutation $\pi\in S_n$ with arcs $A_0,\dots,A_n$. Draw these arcs as semicircles in $H^+$. At each of the points $p_1,\dots,p_n$, separate the two incident endpoints of semicircles into two nearby points of $\ell$ and choose their order so that the two corresponding arcs do not interlace (see an example in Fig.~\ref{fig:cactus plane proof}). The neighborhoods of the original endpoints can be chosen disjoint, so this operation preserves every other interlacement. The endpoints $p_0$ and $p_{n+1}$ are left unchanged.

    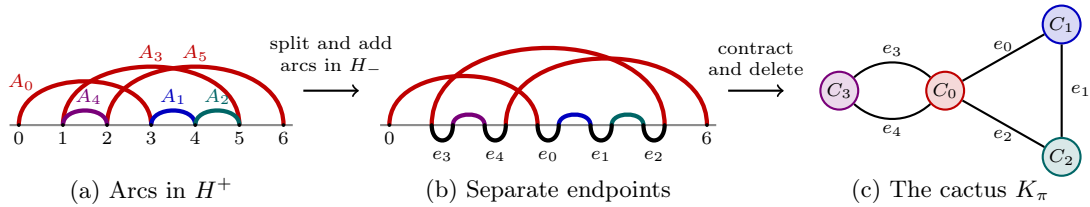
\begin{figure}[h!]
    \centering
    \begin{tikzpicture}[scale = 3.5,
        arc label/.style={font=\scriptsize, inner sep=0.5pt},
        edge label/.style={font=\scriptsize, inner sep=1pt},
        component/.style={circle, draw, thick, minimum size=5mm,
            inner sep=0pt, font=\scriptsize}
    ]
        \colorlet{cactuscolorA}{red!75!black}
        \colorlet{cactuscolorB}{blue!75!black}
        \colorlet{cactuscolorC}{teal!80!black}
        \colorlet{cactuscolorD}{violet!95!black}

        \begin{scope}
            \draw[thick, black!45] (-0.035,0) -- (1.035,0);
            \draw[ultra thick, cactuscolorA] (0,0)
                to[out = 90, in = 90, distance = 6.28319]
                node[arc label, pos=0.24, above left=0pt] {$A_0$} (0.5,0);
            \draw[ultra thick, cactuscolorA] (0.166667,0)
                to[out = 90, in = 90, distance = 8.37758]
                node[arc label, midway, above=1pt] {$A_3$} (0.833333,0);
            \draw[ultra thick, cactuscolorA] (0.333333,0)
                to[out = 90, in = 90, distance = 8.37758]
                node[arc label, midway, above=1pt] {$A_5$} (1,0);
            \draw[ultra thick, cactuscolorB] (0.5,0)
                to[out = 90, in = 90, distance = 2.09440]
                node[arc label, midway, above=1pt] {$A_1$} (0.666667,0);
            \draw[ultra thick, cactuscolorC] (0.666667,0)
                to[out = 90, in = 90, distance = 2.09440]
                node[arc label, midway, above=1pt] {$A_2$} (0.833333,0);
            \draw[ultra thick, cactuscolorD] (0.166667,0)
                to[out = 90, in = 90, distance = 2.09440]
                node[arc label, midway, above=1pt] {$\ A_4$} (0.333333,0);
            \foreach \i in {0,...,6}{
                \fill ({\i/6},0) circle[radius=0.007];
                \node[arc label, below=2pt] at ({\i/6},0) {$\i$};
            }
            \node[font=\small] at (0.5,-0.25) {(a) Arcs in $H^+$};
        \end{scope}

        \draw[thick, ->] (1.075,0.13) --
            node[font=\scriptsize, align=center, above=2pt] {split and add\\arcs in $H_-$} (1.285,0.13);

        \begin{scope}[xshift=1.4cm]
            \draw[thick, black!45] (-0.035,0) -- (1.235,0);
            \draw[ultra thick, cactuscolorA] (0,0)
                to[out = 90, in = 90, distance = 7.03717] (0.56,0);
            \draw[ultra thick, cactuscolorA] (0.16,0)
                to[out = 90, in = 90, distance = 11.05841] (1.04,0);
            \draw[ultra thick, cactuscolorA] (0.44,0)
                to[out = 90, in = 90, distance = 9.55044] (1.2,0);
            \draw[ultra thick, cactuscolorB] (0.64,0)
                to[out = 90, in = 90, distance = 1.50796] (0.76,0);
            \draw[ultra thick, cactuscolorC] (0.84,0)
                to[out = 90, in = 90, distance = 1.50796] (0.96,0);
            \draw[ultra thick, cactuscolorD] (0.24,0)
                to[out = 90, in = 90, distance = 1.50796] (0.36,0);
            \foreach \x/\i in {0.2/3,0.4/4,0.6/0,0.8/1,1/2}{
                \draw[ultra thick] (\x-0.04,0)
                    to[out = -90, in = -90, distance = 2.2]
                    node[edge label, midway, below=2pt] {$e_{\i}$} (\x+0.04,0);
                \fill (\x-0.04,0) circle[radius=0.007];
                \fill (\x+0.04,0) circle[radius=0.007];
            }
            \fill (0,0) circle[radius=0.007];
            \fill (1.2,0) circle[radius=0.007];
            \node[arc label, below=2pt] at (0,0) {$0$};
            \node[arc label, below=2pt] at (1.2,0) {$6$};
            \node[font=\small] at (0.6,-0.25) {(b) Separate endpoints};
        \end{scope}

        \draw[thick, ->] (2.68,0.13) --
            node[font=\scriptsize, align=center, above=2pt]
            {contract\\and delete} (2.89,0.13);

        \begin{scope}[xshift=3.1cm]
            \node[component, draw=cactuscolorD, fill=cactuscolorD!15]
                (C3) at (0,0.13) {$C_3$};
            \node[component, draw=cactuscolorA, fill=cactuscolorA!15]
                (C0) at (0.4,0.13) {$C_0$};
            \node[component, draw=cactuscolorB, fill=cactuscolorB!15]
                (C1) at (0.84,0.38) {$C_1$};
            \node[component, draw=cactuscolorC, fill=cactuscolorC!15]
                (C2) at (0.84,-0.12) {$C_2$};
            \draw[thick] (C0) -- node[edge label, above=1pt] {$e_0$} (C1);
            \draw[thick] (C1) -- node[edge label, right=2pt] {$e_1$} (C2);
            \draw[thick] (C2) -- node[edge label, below=1pt] {$e_2$} (C0);
            \draw[thick] (C3) to[bend left=42]
                node[edge label, above=1pt] {$e_3$} (C0);
            \draw[thick] (C3) to[bend right=42]
                node[edge label, below=1pt] {$e_4$} (C0);
            \node[font=\small] at (0.42,-0.25) {(c) The cactus $K_\pi$};
        \end{scope}
    \end{tikzpicture}
    \caption{Illustration of the proof of Theorem~\ref{thm:cactus} for $\pi=(3,4,5,1,2)$.}
    \label{fig:cactus plane proof}
\end{figure}

    Consider the union of the resulting semicircles as a plane graph whose vertices are the endpoints and intersection points of the semicircles. The connected components of this graph correspond precisely to the connected components of $G_\pi$. Join each pair of separated endpoints by a small semicircle in $H^-$. These $n$ new edges are pairwise disjoint, and the lower side of every one is incident to the unbounded face (see Fig.~\ref{fig:cactus plane proof}(b)).

    Contract a spanning tree in each upper component and delete all its remaining edges, retaining every lower edge, including those which become loops. The resulting plane multigraph is exactly $K_\pi$: its vertices are the components of $G_\pi$, and its edges join the components containing consecutive arcs. Each retained edge is still incident to the common unbounded face. Indeed, edge contraction shortens face boundaries, while edge deletion merges faces, so these operations preserve the indicated incidence.

    A connected plane multigraph whose edges are all incident to one face is a cactus. Otherwise, two simple cycles sharing an edge would yield three internally disjoint paths with the same two endpoints. Delete all other edges. Every remaining edge is still incident to a common face, but a plane graph consisting of three such paths has three faces, each bounded by only two of the paths. This is a contradiction. The argument allows paths consisting of single edges, and hence includes parallel edges, while a loop cannot belong to two distinct simple cycles. Since $K_\pi$ is connected, the theorem follows.
\end{proof}

\begin{definition}\label{def:cut space}
    Let $K$ be a finite multigraph, and let $V(K)$ and $E(K)$ denote its sets of vertices and edges, respectively. Its coboundary map
    $$
    \partial_K\colon\F_2^{V(K)}\longrightarrow\F_2^{E(K)}
    $$
    is defined by $(\partial_Kx)_e=x_u+x_v$ for every edge $e$ with endpoints $u$ and $v$.
    The \emph{cut space} of $K$ is
    $$
    \operatorname{Cut}(K):=\operatorname{im}\partial_K\subseteq\F_2^{E(K)}.
    $$
\end{definition}

\begin{theorem}\label{thm:cut space}
    The realizations of an s-meandric permutation $\pi$ naturally form an affine space over $\mathbb{F}_2$ with translation space $\operatorname{Cut}(K_\pi)$. Choosing a base realization gives an affine bijection with $\operatorname{Cut}(K_\pi)$.
\end{theorem}

\begin{proof}
    Let $\pi \in S_n$ be an s-meandric permutation. Choose a singular meander $M$ realizing $\pi$ as a base realization.
    Under the correspondence used in the proof of Theorem~\ref{thm:degree of realizability}, $M$ corresponds to a proper $2$-coloring
    $
    \varepsilon=(\varepsilon_0,\varepsilon_1,\dots,\varepsilon_n)\in\F_2^{n+1}
    $
    of the interlacement graph $G_\pi$ with $\varepsilon_0=1$.

    Every proper $2$-coloring $\varepsilon'$ with $\varepsilon'_0=1$ is obtained by flipping the colors on some subset of the connected components of $G_\pi$. Thus, there is a vector $x \in \F_2^{V(K_\pi)}$ with $x_{c_0}=0$ such that
    $$
    \varepsilon'_i=\varepsilon_i+x_{c_i}, \qquad i=0,\dots,n,
    $$
    where $c_i$ is the connected component of $G_\pi$ containing $A_i$.

    For each $i=0,\dots,n-1$, let $e_i$ be the edge of $K_\pi$ corresponding to the consecutive arcs $A_i,A_{i+1}$. Then
    $$
    (\partial_{K_\pi}x)_{e_i}=x_{c_i}+x_{c_{i+1}}.
    $$
    Hence the realization corresponding to $\varepsilon'$ determines the element $\partial_{K_\pi}x\in\operatorname{Cut}(K_\pi)$.

    Since $K_\pi$ is connected, $\ker\partial_{K_\pi}$ consists precisely of the constant vectors. Therefore, every element of $\operatorname{Cut}(K_\pi)$ has a unique preimage $x$ satisfying $x_{c_0}=0$. Thus, the above correspondence is a bijection between the realizations of $\pi$ and $\operatorname{Cut}(K_\pi)$.

    Changing the base realization by a component-flip vector $x_0$ replaces the coordinate $\partial_{K_\pi}x$ by
    $$
    \partial_{K_\pi}(x+x_0) = \partial_{K_\pi}x+\partial_{K_\pi}x_0.
    $$
    Thus the bijections obtained from different base realizations differ by translations and define the same affine-space structure.
\end{proof}

\begin{remark}
    The bijection is not canonical. However, for a meandric permutation, we may take the unique non-singular realization of $\pi$ as the base realization, which makes the bijection canonical.
\end{remark}

\subsection{Signed component spine}
Let $\mathcal{R}(\pi)$ be the set of singular meanders realizing an s-meandric permutation $\pi\in S_n$.
For $i=0,1,\dots,n$, denote by $c_i$ the connected component of the interlacement graph containing the arc $A_i$, as in Definition~\ref{def:component spine}. For $i=0,1,\dots,n-1$, let $e_i$ be the edge of $K_\pi$ corresponding to the consecutive arcs $A_i,A_{i+1}$. It joins $c_i$ and $c_{i+1}$ and corresponds to the intersection $p_{\pi(i+1)}$.
For $M\in\mathcal{R}(\pi)$, define its \emph{tangency vector} $\tau_M\in\F_2^{E(K_\pi)}$ by
$$
(\tau_M)_e=
\begin{cases}
1, & \text{if the intersection corresponding to the edge $e$ is a touch in $M$},\\
0, & \text{otherwise}.
\end{cases}
$$
Thus, $\operatorname{t}(M)$ is the Hamming weight of $\tau_M$.

Fix a base realization $M_0\in\mathcal{R}(\pi)$.
Under the bijection constructed in the proof of Theorem~\ref{thm:cut space}, flipping the colors on the connected components according to a vector $\varphi\in\F_2^{V(K_\pi)}$ with $\varphi_{c_0}=0$ changes the tangency coordinate corresponding to $e_i$ by
$$
    \varphi_{c_i}+\varphi_{c_{i+1}} = (\partial_{K_\pi}\varphi)_{e_i}.
$$
Since adding a constant vector does not change the coboundary, every vector has the same coboundary as one satisfying $\varphi_{c_0}=0$. Hence the set of tangency vectors of all realizations of $\pi$ is
$$
    \tau_{M_0}+\operatorname{Cut}(K_\pi)
    =
    \left\{
    \tau_{M_0}+\partial_{K_\pi}\varphi
    \mathrel{}\middle|\mathrel{}
    \varphi\in\F_2^{V(K_\pi)}
    \right\}.
$$

For a realization $M\in\mathcal{R}(\pi)$, define a \emph{signature} $J_M\colon E(K_\pi)\longrightarrow\{-1,1\}$ by
$$
    J_M(e)=(-1)^{(\tau_M)_e}.
$$
Thus, an edge is negative precisely if its corresponding intersection is a touch. If $M'$ is another realization, then $\tau_M$ and $\tau_{M'}$ belong to the same affine coset of $\operatorname{Cut}(K_\pi)$. Hence there exists $\varphi\in\F_2^{V(K_\pi)}$ such that
$$
    \tau_{M'}=\tau_M+\partial_{K_\pi}\varphi.
$$
Define a \emph{switching function} $s\colon V(K_\pi) \to \{-1,1\}$ by
$$
    s(v):=(-1)^{\varphi_v}.
$$
We obtain
$$
    J_{M'}(e)=s(u)J_M(e)s(v)
$$
for every edge $e$ in $K_\pi$ with endpoints $u$ and $v$. Therefore, all pairs $(K_\pi,J_M)$ are switching equivalent.

\begin{definition}\label{def:signed component spine}
    For an s-meandric permutation $\pi$, the common switching class of the pairs $(K_\pi,J_M)$ for $M \in \mathcal{R}(\pi)$, denoted by $\Sigma_\pi$, is called the \emph{signed component spine} of $\pi$.
\end{definition}

For a cycle $C$ of $K_\pi$, define its sign in $\Sigma_\pi$ by
$$
\prod_{e\in E(C)}J(e),
$$
where $(K_\pi,J)$ is any representative of $\Sigma_\pi$. This sign is independent of the representative. A cycle is called \emph{positive} or \emph{negative} according to its sign, and the signed component spine is called \emph{balanced} if every cycle is positive.

\begin{remark}
    The notions of switching and balance used here are standard in signed-graph theory; see \cite[Sections~2--3]{Zaslavsky82SignedGraphs}. For a fixed underlying graph, encoding positive edges by $0$ and negative edges by $1$ identifies switching classes with affine cosets of the binary cut space; see Sol\'e and Zaslavsky~\cite[Lemma~1]{SoleZaslavsky94CodingApproachSignedGraphs}. In our setting, the coset
    $$
        \tau_{M_0}+\operatorname{Cut}(K_\pi) = \{\tau_M \mid M\in\mathcal{R}(\pi)\}
    $$
    is therefore the binary encoding of the signed component spine $\Sigma_\pi$. Although \cite{SoleZaslavsky94CodingApproachSignedGraphs} formulates this correspondence for loopless graphs, it extends directly to the multigraphs considered here: switching leaves loop signs unchanged, and every cut vector has zero coordinates on loops.
\end{remark}

\section{Tangency polynomial}\label{sec:tangency polynomial}
\begin{definition}
    Let $\mathcal{R}(\pi)$ denote the set of singular meanders realizing a permutation $\pi$. The \emph{tangency polynomial} of $\pi$ is
    $$
    T_\pi(x)=\sum_{M \in \mathcal{R}(\pi)}x^{\operatorname{t}(M)}.
    $$
\end{definition}

\begin{convention}
    Throughout this paper, we denote the tangency polynomial of a permutation $\pi$ by $T_\pi(x)$.
\end{convention}

For an s-meandric permutation, the representatives of $\Sigma_\pi$ are in bijection with its realizations, and hence
$$
T_\pi(x)=\sum_{(K_\pi,J)\in\Sigma_\pi}x^{|E^-(J)|},
$$
where $E^-(J)=\{e\in E(K_\pi)\mid J(e)=-1\}$. Thus, $T_\pi(x)$ is the negative-edge weight enumerator of the signed component spine.

Many properties of the realizations of a permutation $\pi$ can be read directly from its tangency polynomial:
\begin{enumerate}
    \item $\pi$ is s-meandric if and only if $T_\pi(x)\not\equiv 0$;
    \item $\pi$ is meandric if and only if $T_\pi(0)=1$;
    \item the degree of realizability $d(\pi)$ is $T_\pi(1)$.
\end{enumerate}
For an s-meandric permutation $\pi$, the following also hold:
\begin{enumerate}
    \setcounter{enumi}{3}
    \item the \emph{minimum tangency number} $t_-(\pi)$ is the minimum number of touches among the realizations of $\pi$; it is equal to $\min \operatorname{supp} T_\pi$, where $\operatorname{supp} T_\pi$ denotes the set of exponents with nonzero coefficient in $T_\pi$;
    \item the \emph{maximum tangency number} $t_+(\pi)$ is the maximum number of touches among the realizations of $\pi$; it is equal to $\deg T_\pi$;
    \item the mean number of touches in a uniformly random realization is equal to $T_\pi'(1)/T_\pi(1)$.
\end{enumerate}

\begin{example}
Some tangency polynomials for permutations in $S_6$ are given below.
$$
{
\renewcommand{\arraystretch}{1.1}
\begin{array}{c|c}
\pi & T_\pi(x) \\
\hline
(2,4,1,3,5,6) & 0 \\
(1,2,3,4,5,6) & (1+x)^6 \\
(3,6,4,2,5,1) & x^5 \\
(3,2,1,6,5,4) & 1+6x^2+9x^4 \\
(4,5,6,2,3,1) & 2x+6x^3 \\
(6,3,4,5,1,2) & 2x+6x^3 \\
\end{array}
}
$$
\end{example}

\subsection{Block factorization theorem}\label{sec:block factorization}
Let $\pi\in S_n$ be an s-meandric permutation, fix a base realization $M_0\in\mathcal{R}(\pi)$, and let $J=J_{M_0}$ be the corresponding signature. Then $(K_\pi,J)$ represents the signed component spine $\Sigma_\pi$.
Let $B$ be a cycle block of $K_\pi$ (possibly a loop), and let $E(B)$ be the set of its edges. Define
$$
\lambda_{B}:=|E(B)|
\qquad\text{and}\qquad
\eta_{B}:=\sum_{e\in E(B)}(\tau_{M_0})_e\in\F_2,
$$
where we identify $\F_2$ with $\{0,1\}$. Note that $\eta_{B}$ is independent of the choice of the base realization $M_0$, since the restriction of every cut to a cycle has even weight. We call $\eta_B$ the \emph{frustration parity} of $B$. The sign of $B$, given by
$$
\prod_{e\in E(B)}J(e)=(-1)^{\eta_B},
$$
is invariant under switching. Thus, $B$ is positive if $\eta_B=0$ and negative if $\eta_B=1$. A negative loop is called a \emph{frustrated loop}. The intersection corresponding to a frustrated loop is a touch in every realization, whereas the intersection corresponding to a positive loop is a transverse crossing in every realization.

The restrictions of the tangency vectors of all realizations to $E(B)$ are precisely the vectors $z\in\F_2^{E(B)}$ satisfying
$$
\sum_{e\in E(B)}z_e=\eta_{B}.
$$
Indeed, the restriction of $\operatorname{Cut}(K_\pi)$ to $E(B)$ is the cut space of $B$, which consists precisely of the vectors of even weight. Therefore, a cycle block of length $\lambda$ and frustration parity $\eta$ contributes
$$
E_{\lambda,\eta}(x)
:=
\sum_{\substack{0\leq j\leq\lambda\\
                 j\equiv\eta\;(\operatorname{mod}2)}}
\binom{\lambda}{j}x^j
=
\frac{(1+x)^\lambda+(-1)^\eta(1-x)^\lambda}{2}
$$
to the tangency polynomial of $\pi$. In particular, $E_{1,0}(x)=1$ and $E_{1,1}(x)=x$.

On a bridge, the tangency coordinate can take either value. By Theorem~\ref{thm:cactus}, the blocks of $K_\pi$ form a tree-like arrangement. Local component-flip vectors on the blocks can therefore be made to agree at shared vertices by successively adding constant vectors, which does not change their coboundaries. Thus, the choices on distinct blocks are independent. Each bridge contributes a factor $1+x$, and each cycle block contributes the corresponding factor $E_{\lambda_B,\eta_B}(x)$. This gives the following result.

\begin{theorem}\label{thm:block factorization}
    Let $\pi\in S_n$ be an s-meandric permutation. Suppose that $K_\pi$ has $b$ bridges. Then the tangency polynomial of $\pi$ factors as
    \begin{equation}\label{eq:block factorization}
        T_\pi(x)
        =
        (1+x)^b
        \prod_{B} E_{\lambda_{B},\eta_{B}}(x),
    \end{equation}
    where the product is taken over all cycle blocks (including loops) $B$ of $K_\pi$.
\end{theorem}

\begin{example}
    Consider the permutation $\pi=(3,2,1,6,5,4,9,8,7,11,12,10)$ from Example~\ref{ex:component spine}. Take the singular meander shown in Fig.~\ref{fig:meander example for block factorization} as a base realization of $\pi$. The component spine has no bridges; its three $3$-cycle blocks are positive, its $2$-cycle block is negative, and its single loop is negative. Thus
    $$
    T_\pi(x) = E_{1,1}(x)E_{2,1}(x)E_{3,0}(x)^3 = x(2x)(1+3x^2)^3 = 2x^2+18x^4+54x^6+54x^8.
    $$
    \begin{figure}[ht]
        \centering
        \begin{tikzpicture}[scale = 10]
        \draw[thick] (-0.05, 0) to (1.05, 0);
        \draw[ultra thick] (0,0)
            to[out = 90, in = 90, distance = 2.89993] (0.230769, 0)
            to[out = 90, in = 90, distance = 0.966644] (0.153846, 0)
            to[out = 90, in = 90, distance = 0.966644] (0.0769231, 0)
            to[out = -90, in = -90, distance = 4.83322] (0.461538, 0)
            to[out = -90, in = -90, distance = 0.966644] (0.384615, 0)
            to[out = -90, in = -90, distance = 0.966644] (0.307692, 0)
            to[out = 90, in = 90, distance = 4.83322] (0.692308, 0)
            to[out = -90, in = -90, distance = 0.966644] (0.615385, 0)
            to[out = 90, in = 90, distance = 0.966644] (0.538462, 0)
            to[out = -90, in = -90, distance = 3.86658] (0.846154, 0)
            to[out = -90, in = -90, distance = 0.966644] (0.923077, 0)
            to[out = 90, in = 90, distance = 1.93329] (0.769231, 0)
            to[out = 90, in = 90, distance = 2.89993] (1, 0);
        \end{tikzpicture}
        \caption{A singular meander realizing the permutation $(3,2,1,6,5,4,9,8,7,11,12,10)$.}
        \label{fig:meander example for block factorization}
    \end{figure}
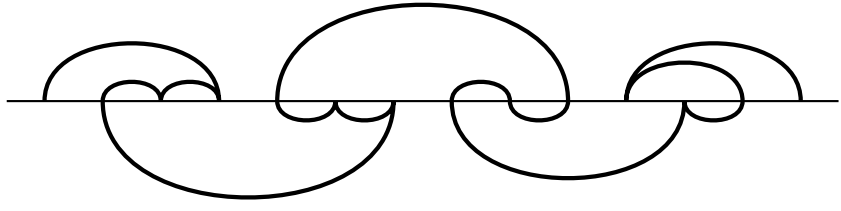
\end{example}

\begin{corollary} \label{cor:meandric permutation balanced spine}
    An s-meandric permutation is meandric if and only if its signed component spine is balanced.
\end{corollary}
\begin{proof}
    By~\eqref{eq:block factorization}, $T_\pi(0)=1$ if and only if $\eta_{B}=0$ for every cycle block $B$.
\end{proof}

\begin{corollary}
    In the notation of Theorem~\ref{thm:block factorization}, the minimum tangency number of $\pi$ is equal to
    \begin{equation}\label{eq:minimum tangency number}
         t_-(\pi) =
        \left|\left\{B\mid B \text{ is a cycle block of } K_\pi
        \text{ and } \eta_{B}=1\right\}\right|.
    \end{equation}
    The number of realizations attaining the minimum is
    $$
    \left[x^{t_-(\pi)}\right]T_\pi(x) = \prod_{B} \lambda_{B},
    $$
    where the product runs over all cycle blocks $B$ of $K_\pi$ with $\eta_{B}=1$.
\end{corollary}

\begin{proof}
    A cycle block $B$ contributes a minimum of $0$ touches if $\eta_{B}=0$ and a minimum of $1$ touch if $\eta_{B}=1$, while each bridge contributes a minimum of $0$ touches; this proves~\eqref{eq:minimum tangency number}. Moreover, for $\eta_{B}=1$, the coefficient of $x$ in $E_{\lambda_{B},1}(x)$ is $\lambda_{B}$, whereas all other factors have constant term $1$.
\end{proof}

\begin{corollary}\label{cor:inequality}
    For every s-meandric permutation $\pi\in S_n$,
    $$
    t_-(\pi)+\log_2 d(\pi)\leq n.
    $$
    Equality holds if and only if every cycle block of the signed component spine is negative.
\end{corollary}

\begin{proof}
    Let $c$ be the number of cycle blocks $B$ with $\eta_{B}=1$. Using~\eqref{eq:block factorization} to evaluate $T_\pi(1)$ gives
    $$
    d(\pi)=2^{b+\sum_{B}(\lambda_{B}-1)}.
    $$
    Since the bridges and cycle blocks partition $E(K_\pi)$, we have
    $$
    n=b+\sum_{B}\lambda_{B}.
    $$
    Thus
    $$
    t_-(\pi)+\log_2 d(\pi)
    =
    c+b+\sum_{B}(\lambda_{B}-1)
    \leq n.
    $$
    Equality holds precisely when every cycle block is counted in $c$, that is, when $\eta_{B}=1$ for every cycle block $B$.
\end{proof}

\begin{corollary}\label{cor:roots of tangency polynomials}
    For an s-meandric permutation $\pi$, every zero of $T_\pi(x)$ lies on the imaginary axis or is equal to $-1$.
\end{corollary}

\begin{proof}
    By Theorem~\ref{thm:block factorization}, it suffices to consider the factors $1+x$ and $E_{\lambda,\eta}(x)$. The first has zero $-1$. If $E_{\lambda,\eta}(z)=0$, then
    $$
    (1+z)^\lambda=(-1)^{\eta+1}(1-z)^\lambda,
    $$
    so $|1+z|=|1-z|$, which is equivalent to $\operatorname{Re}z=0$.
\end{proof}

\subsection{Ising model interpretation of the tangency polynomial}\label{sec:Ising}
The tangency polynomial has a natural interpretation as an Ising partition function (we refer the reader to~\cite{DuminilCopin20IsingPotts} for an introduction to Ising models and to~\cite{Sokal05TuttePotts} for the graph-polynomial interpretation of finite Ising partition functions). Let $\pi$ be an s-meandric permutation, and choose a representative $(K_\pi,J)$ of its signed component spine $\Sigma_\pi$. A spin configuration $\sigma\in\{-1,1\}^{V(K_\pi)}$, considered up to global reversal $\sigma\mapsto-\sigma$, determines a realization of $\pi$. An edge $e$ joining vertices $u$ and $v$ corresponds to a touch if and only if $J(e)\sigma_u\sigma_v=-1$. Therefore, if
$$
Z_{K_\pi}(\beta;J)
:=
\sum_{\sigma\in\{-1,1\}^{V(K_\pi)}}
\exp\left(
    \beta\sum_{e=uv}J(e)\sigma_u\sigma_v
\right)
$$
is the partition function of the zero-field signed Ising model on $K_\pi$, with inverse temperature $\beta\geq0$, then
$$
T_\pi\left(e^{-2\beta}\right)
=
\frac{1}{2}e^{-\beta|E(K_\pi)|}Z_{K_\pi}(\beta;J).
$$
Thus, in Ising model terminology, $T_\pi(x)$ enumerates spin configurations modulo global reversal by their number of frustrated edges. The minimum tangency number $t_-(\pi)$ is the minimum number of negative edges among the representatives of $\Sigma_\pi$, usually called its frustration index. The coefficient $[x^{t_-(\pi)}]T_\pi(x)$ is half the ground-state degeneracy. Changing the representative corresponds to a gauge transformation: for a switching function $s\colon V(K_\pi)\to\{-1,1\}$, we simultaneously replace $J(e)$ by $s(u)J(e)s(v)$ and $\sigma_v$ by $s(v)\sigma_v$. This leaves the partition function unchanged.
    
\subsection{Support and coefficients of the tangency polynomial}
\begin{definition}\label{def:eulerian component spine}
    Let $\pi\in S_n$ be an s-meandric permutation with arcs $A_0,A_1,\dots,A_n$. We say that its component spine $K_\pi$ is \emph{Eulerian} if $A_0$ and $A_n$ correspond to the same vertex in $K_\pi$. We say that it is \emph{non-Eulerian} otherwise.
\end{definition}
By Remark~\ref{rem:component spine has Euler trail}, this agrees with the usual graph-theoretic terminology.

\begin{theorem}\label{thm:polynomial coefficients}
    Let $\pi\in S_n$ be an s-meandric permutation, and let $t_-(\pi)$ and $t_+(\pi)$ be its minimum and maximum tangency numbers, respectively.
    \begin{enumerate}
        \item \label{thm:poly coef case 1} If $K_\pi$ is Eulerian, every exponent in $T_\pi(x)$ is congruent to $t_-(\pi)$ modulo $2$, and every such exponent between $t_-(\pi)$ and $t_+(\pi)$ occurs. The nonzero coefficients of $T_\pi(x)$ form a log-concave sequence.
        \item \label{thm:poly coef case 2} If $K_\pi$ is non-Eulerian, every integer between $t_-(\pi)$ and $t_+(\pi)$ occurs as an exponent in $T_\pi(x)$. The coefficients of $T_\pi(x)$ form a unimodal sequence.
    \end{enumerate}
\end{theorem}

\begin{proof}
    As noted in Remark~\ref{rem:component spine has Euler trail}, a component spine has an Euler trail. Delete all loops and collapse each non-loop cycle block to a vertex, deleting its edges. The resulting graph is a tree whose edges correspond precisely to the bridges of $K_\pi$, and the Euler trail projects to an Euler trail that uses every edge of this tree exactly once. A tree admitting such a trail is a path, possibly consisting of a single vertex. The latter occurs exactly when $K_\pi$ has no bridges. Since cycle blocks contribute even degrees at every vertex and an Eulerian graph has no bridges, this is equivalent to $K_\pi$ being Eulerian. Thus, an Eulerian component spine has no bridges, while for a non-Eulerian component spine the resulting tree is a nontrivial path.

    For every integer $\lambda \geq 1$ and $\eta\in\{0,1\}$, we can write
    $
    E_{\lambda,\eta}(x)
    =
    x^\eta P_{\lambda,\eta}(x^2),
    $
    where the coefficients of $P_{\lambda,\eta}$ are
    $$
    \binom{\lambda}{\eta},
    \binom{\lambda}{\eta+2},
    \binom{\lambda}{\eta+4},
    \dots
    $$
    up to the largest index not exceeding $\lambda$. These coefficients form a positive log-concave sequence.

    If $K_\pi$ is Eulerian, it has no bridges and hence, by Theorem~\ref{thm:block factorization},
    $
    T_\pi(x)=x^{t_-(\pi)}Q(x^2),
    $
    where the coefficients of $Q$ are obtained by convolving the coefficient sequences of the polynomials $P_{\lambda_{B},\eta_{B}}$. Since convolution preserves log-concavity for nonnegative log-concave sequences without internal zeros (see, for example,~\cite{Hoggar1974Chromatic}), the coefficients of $Q$ are positive and log-concave. This proves case~\ref{thm:poly coef case 1}.

    Suppose now that $K_\pi$ is non-Eulerian. Then it has at least one bridge, so the factorization of $T_\pi(x)$ contains a factor $1+x$. Before taking the bridge factors into account, the polynomial has the form
    $
    x^{t_-(\pi)}Q(x^2),
    $
    where the coefficients of $Q$ form a positive log-concave, and hence unimodal, sequence. Multiplication by one factor $1+x$ repeats this coefficient sequence in adjacent degrees, so the support becomes a full interval and the resulting coefficient sequence is unimodal. Multiplication by each remaining factor $1+x$ replaces a coefficient sequence by its adjacent-sum sequence, which preserves unimodality. This proves case~\ref{thm:poly coef case 2}.
\end{proof}

\begin{example}
    If the component spine is non-Eulerian, log-concavity of the full coefficient sequence can fail. For example, for $\pi=(1,4,3,2)$ we have
    $$
    T_{\pi}(x)=1+x+3x^2+3x^3.
    $$
\end{example}

\subsection{Symmetries and operations}
\begin{theorem}\label{thm:symmetries}
    For $\pi\in S_n$,
    $$
    T_\pi(x)=T_{\pi^{-1}}(x)=T_{\pi^r}(x),
    $$
    where $\pi^r$ is the reverse-complement permutation defined by
    $$
    \pi^r(i)=n+1-\pi(n+1-i).
    $$
\end{theorem}

\begin{proof}
    Let $M$ be a singular meander realizing $\pi$. Reflecting $\R^2$ across the line $\left\{\left(\frac{n+1}{2}, y\right) \mid y\in \R\right\}$ and reversing the direction of traversal gives, after a possible reflection across $\ell$, a singular meander realizing $\pi^r$. This operation preserves the number of touches and is an involution.

    To obtain $\pi^{-1}$, interchange the roles of $M$ and the segment $I=\{(x,0)\mid x\in[0;n+1]\}$. To normalize the resulting configuration, extend $M$ along the two exterior rays of $\ell$ and straighten the resulting proper simple curve to $\ell$ by an ambient isotopy, placing the marked intersection points at $p_0,\dots,p_{n+1}$ in their order along $M$. The subarcs of the image of $I$ have noninterlacing endpoints within each half-plane, so they can be replaced by the corresponding semicircles. Reflect across $\ell$ if necessary to ensure the normalization at $p_0$.

    The orders of the intersection points along the two curves are interchanged, so the resulting permutation is $\pi^{-1}$. These operations preserve the touch type at every intersection. For a fixed permutation, the ordered touch types and the initial normalization determine the half-plane of every semicircle. Thus, the normalized realization is independent of the normalization choices, and interchanging the curves twice recovers $M$. This gives the required touch-preserving bijection.
\end{proof}

\begin{definition}
    Let $\pi=(a_1,a_2,\dots,a_n)\in S_n$ and $\tau=(b_1,b_2,\dots,b_m)\in S_m$. Their \emph{direct sum} is the permutation $\pi\oplus\tau\in S_{n+m}$ defined by
    $$
    \pi\oplus\tau
    =
    (a_1,a_2,\dots,a_n,n+b_1,n+b_2,\dots,n+b_m).
    $$
\end{definition}

\begin{theorem}\label{thm:poly sum of permutations}
    For $\pi\in S_n$ and $\tau\in S_m$,
    $$
    T_{\pi\oplus\tau}(x)=T_\pi(x)T_\tau(x).
    $$
\end{theorem}

\begin{proof}
    We construct a bijection
    $$
    Q\colon\mathcal{R}(\pi)\times\mathcal{R}(\tau)
    \longrightarrow
    \mathcal{R}(\pi\oplus\tau)
    $$
    such that
    $$
    \operatorname{t}(Q(M_\pi,M_\tau))
    =
    \operatorname{t}(M_\pi)+\operatorname{t}(M_\tau).
    $$
    Let $M_\pi$ and $M_\tau$ be singular meanders realizing $\pi$ and $\tau$, respectively. Let $s_n$ be the semicircle of $M_\pi$ joining $(a_n,0)$ to $(n+1,0)$, and let $s'_0$ be the semicircle of $M_\tau$ joining $(0,0)$ to $(b_1,0)$.

    Remove the interiors of these two semicircles, together with the endpoints $(n+1,0)$ and $(0,0)$, respectively. Translate the remaining part of $M_\tau$ by $n$ units to the right and, if necessary, reflect it across $\ell$ so that the half-plane containing the interior of $s'_0$ is mapped to the half-plane containing the interior of $s_n$. Since the half-plane changes at each of the $n-\operatorname{t}(M_\pi)$ transverse crossings of $M_\pi$, the required affine transformation is
    $$
    F(x,y)=\left(x+n,(-1)^{n-\operatorname{t}(M_\pi)}y\right).
    $$
    Finally, join $(a_n,0)$ to $(n+b_1,0)$ by a semicircle $C$ whose interior lies in the same half-plane as that of $s_n$.

    Within either summand, replacing the removed semicircle by $C$ does not change which retained semicircles have interlacing endpoints with it. Therefore, $C$ introduces no additional intersections. The resulting singular meander $Q(M_\pi,M_\tau)$ realizes $\pi\oplus\tau$. The local type of every intersection is preserved, including those at $(a_n,0)$ and $(n+b_1,0)$, and hence
    $$
    \operatorname{t}(Q(M_\pi,M_\tau))
    =
    \operatorname{t}(M_\pi)+\operatorname{t}(M_\tau).
    $$

    To reverse the construction, remove $C$ and restore a terminal semicircle from $(a_n,0)$ to $(n+1,0)$ for the left part and an initial semicircle from $(n,0)$ to $(n+b_1,0)$ for the right part, both on the same side of $\ell$ as $C$. Translate the right part back by $n$ units and reflect it across $\ell$ if necessary to normalize its initial arc. This recovers the two realizations and proves bijectivity. Summing over all pairs of realizations yields the statement of the theorem.
\end{proof}

\subsection{Tangencies in a random realization}
\begin{theorem}\label{thm:central limit theorem}
    Let $\pi \in S_n$ be s-meandric, and let $M\in\mathcal{R}(\pi)$ be chosen uniformly. Set $X_\pi:=\operatorname{t}(M)$ and $\Delta_\pi:=t_+(\pi)-t_-(\pi)$. Then
    $$
    \frac{\Delta_\pi}{4}\leq \operatorname{Var}X_\pi\leq\frac{\Delta_\pi}{2}.
    $$
    Moreover, there is an absolute constant $C$ such that, whenever $\Delta_\pi>0$,
    $$
    \sup_{z\in\R}\left|
        \mathbb{P}\left(\frac{X_\pi-\mathbb{E}X_\pi}{\sqrt{\operatorname{Var}X_\pi}}\leq z\right)-\Phi(z)
    \right|\leq\frac{C}{\sqrt{\Delta_\pi}},
    $$
    where $\Phi$ is the standard normal distribution function.
\end{theorem}

\begin{proof}
    By Theorem~\ref{thm:block factorization}, $X_\pi$ can be written as a sum of independent contributions from the blocks and bridges. Here, the range of a random variable means the difference between its largest and smallest possible values. A bridge contributes a fair Bernoulli variable, a loop contributes a constant, and a cycle of length $\lambda\geq2$ contributes a binomial random variable $\operatorname{Bin}(\lambda,1/2)$ conditioned to have parity $\eta$. For $\lambda \geq 3$, this last variable has variance $\lambda/4$ and range $2\lfloor(\lambda-\eta)/2\rfloor$. A positive $2$-cycle has variance $1$ and range $2$, whereas a negative $2$-cycle is constant. Thus each contribution has variance between one quarter and one half of its range. Adding the variances and ranges proves the first assertion.

    Corollary~\ref{cor:roots of tangency polynomials} gives a second decomposition, into bounded independent variables:
    $$
    X_\pi\overset{\mathrm d}{=}
    t_-(\pi)+\sum_{i=1}^{b}B_i+2\sum_{j=1}^{r}C_j,
    $$
    where $b$ is the number of bridges. Indeed, after dividing $T_\pi(x)$ by $x^{t_-(\pi)}(1+x)^b$, the remaining polynomial is $Q(x^2)$, where $Q(0)>0$ and all roots of $Q$ are strictly negative. Factoring $Q$ and normalizing it at $1$ gives Bernoulli variables $C_j$, with $r=\deg Q$; the $B_i$ are fair Bernoulli variables. Every centered summand $Y$ has absolute value at most $2$, so $\mathbb{E}|Y|^3\leq2\mathbb{E}Y^2$. Berry's inequality~\cite[Theorem~1]{Berry41GaussianApproximation} therefore bounds the error of the normal approximation by $2C_0/\sqrt{\operatorname{Var}X_\pi}$ for an absolute constant $C_0$. Combining this with the lower variance bound proves the estimate with $C=4C_0$.
\end{proof}

\begin{remark}
    Since each non-loop edge has tangency probability $1/2$, the expectation of $X_\pi$ depends only on $n$ and the numbers $l_+$ and $l_-$ of positive and negative loops, respectively:
    $$
    \mathbb{E} X_\pi = \frac{n + l_- - l_+}{2}.
    $$
\end{remark}

\begin{remark}
    Consequently, along any sequence of s-meandric permutations with positive variances, the standardized tangency counts converge to the standard normal law if and only if $\Delta_\pi\to\infty$. Indeed, if $\Delta_\pi$ fails to tend to infinity, there is a subsequence on which $\Delta_\pi \leq D$ for some positive integer $D$. On this subsequence, the standardized tangency count has at most $D+1$ possible values, so some value has measure at least $1/(D+1)$. Its distribution function therefore differs from every continuous distribution function by at least $1/(2(D+1))$ in supremum norm, precluding convergence to a continuous law.

    The range condition cannot be replaced by growth of $n$ or $d(\pi)$: for $\pi=(2,1)\oplus\cdots\oplus(2,1)$ with $r$ summands, $T_\pi(x)=(2x)^r$, so all $2^r$ realizations have exactly $r$ touches. Appending a further summand $(1)$ gives tangency polynomial $(2x)^r(1+x)$ and variance $1/4$. The standardized count then takes the values $-1$ and $1$ with equal probabilities, although both the order and the degree of realizability tend to infinity.
\end{remark}

\section{Asymptotics and enumeration}\label{sec:asymptotics and enumeration}
    As noted in the introduction, the precise asymptotic behavior of the numbers $m_n$ of meandric permutations in $S_n$ is not known. It is conjectured that
    $$
    m_n\sim C_{\text{even}} R^n n^{-\alpha} \quad \text{for even } n, \qquad
    m_n\sim C_{\text{odd}} R^n n^{-\alpha} \quad \text{for odd } n,
    $$
    as $n\to\infty$, where $C_{\text{even}}$ and $C_{\text{odd}}$ are positive constants and $\alpha=\frac{29+\sqrt{145}}{12}$. The value of $\alpha$ was conjectured by Di Francesco, Golinelli, and Guitter~\cite{DiFrancescoGolinelliGuitter2000Asymptotics}. The best known rigorous bounds for the constant $R$ are
    $$
    3.3734\leq R\leq 3.5918,
    $$
    as proved by Albert and Paterson~\cite{AlbertPaterson05GrowthRate}, while the numerical computations of Jensen and Guttmann~\cite{JensenGuttmann2000CriticalExponents} suggest $R=3.501837(3)$. At present, the numbers $m_n$ are known up to $n=55$: Jensen computed them up to $n=43$ in~\cite{Jensen2000Transfer}, and the values for $44\leq n\leq55$ are due to Howroyd (see~\cite[\oeislink{A005316}]{oeis}).

    The enumeration of s-meandric permutations appears to be similarly difficult. Let $s_n$ denote the number of s-meandric permutations in $S_n$. We have computed these numbers up to $n=24$. Let us call a permutation $\pi=(\pi(1),\dots,\pi(n))$ \emph{sum-indecomposable} if there is no integer $k$ with $1\leq k<n$ such that $\{\pi(1),\dots,\pi(k)\}=\{1,\dots,k\}$. Once the numbers $g_n$ of sum-indecomposable s-meandric permutations in $S_n$ have been obtained, the total numbers $s_n$ are recovered from the unique direct-sum decomposition and Theorem~\ref{thm:poly sum of permutations}:
    \begin{equation}\label{eq:s-meandric numbers from sum-indecomposable}
        s_0=1,\qquad
        s_n=\sum_{k=1}^{n}g_k s_{n-k}\quad(n\geq1).
    \end{equation}
    We have computed the numbers $g_1,\dots,g_{24}$; see Appendix~\ref{appendix}. The values were computed by a frontier transfer-matrix algorithm in the spirit of Jensen~\cite{Jensen2000Transfer}, augmented with interlacement-component tracking and canonical colorings to count each permutation once. A detailed description of the algorithm, the code, and implementation details are available in the accompanying repository~\cite{Bcode}.

    Now we show that the numbers $s_n$ have a well-defined exponential growth rate and give bounds for it.
    \begin{theorem}\label{thm:estimate on s-meandric number}
        Let $s_n$ be the number of s-meandric permutations in $S_n$. Then the limit
        $$
        \mu:=\lim_{n\to\infty}\sqrt[n]{s_n}
        $$
        exists. Moreover,
        $$
        5.048565\leq\mu\leq16.
        $$
    \end{theorem}
    
    \begin{proof}
        \textbf{Existence of the limit.} If $\pi\in S_n$ and $\tau\in S_m$ are s-meandric permutations, then $\pi\oplus\tau\in S_{n+m}$ is also s-meandric (see Theorem~\ref{thm:poly sum of permutations}). Thus, $s_{n+m}\geq s_n s_m$.
        Therefore, the sequence is supermultiplicative and thus, by Fekete's lemma (see, for example,~\cite{Van2001combinatorics}), the limit
        $
        \lim_{n\to\infty}\sqrt[n]{s_n}
        $
        exists.
        
        \textbf{The upper bound.} We first prove that $s_n\leq C_{n+1}C_{n+2}$, where $C_m=\frac{1}{m+1}\binom{2m}{m}$
        is the Catalan number.
        Let $\pi\in S_n$ be an s-meandric permutation. Consider a singular meander $M$ realizing it. Then $M\setminus\ell$ (recall Definition~\ref{def:singular meanders}) is a union of $n+1$ semicircles, each of which lies either in $H^+$ or in $H^-$.
        By resolving each common endpoint of the semicircles into distinct nearby points of $\ell$, we obtain a perfect matching on $2(n+1)$ points of $\ell$, drawn with each edge in $H^+$ or in $H^-$ and with no crossings within either half-plane. Such a matching uniquely determines the permutation. Choosing one realization for each s-meandric permutation therefore gives an injection into the set of such matchings. The number of these matchings is
        $$
            \sum_{k=0}^{n+1}\binom{2n+2}{2k}C_kC_{n+1-k} = C_{n+1}C_{n+2}.
        $$
        Indeed, there are $\binom{2n+2}{2k}$ ways to choose the $2k$ points that are joined by semicircles in $H^+$, and there are $C_k$ and $C_{n+1-k}$ ways to form a noncrossing matching in $H^+$ and in $H^-$, respectively. Hence $s_n\leq C_{n+1}C_{n+2}$.
        Since
        $$
            \lim_{n\to\infty}\sqrt[n]{C_{n+1}C_{n+2}}=16,
        $$
        we obtain the upper bound $\mu\leq16$.
    
        \textbf{The lower bound.} For the lower bound, we use the computed values $g_1,\dots,g_{24}$ and the fact that $g_{n+1} \geq 2g_{n}$ for $n \geq 2$, which is proved in Lemma~\ref{lem:bound on gn} below. Consider the sequence $\{\hat{g}_n\}_{n \geq 1}$ defined by
        $$
        \hat{g}_n=
        \begin{cases}
            g_n, & n\leq24,\\
            2\hat{g}_{n-1}, & n>24,
        \end{cases}
        $$
        and let
        $$
        \hat{G}(z)
        =
        \sum_{n \geq 1}\hat{g}_nz^n
        =
        \sum_{n=1}^{24}g_nz^n+\frac{2g_{24}z^{25}}{1-2z}
        $$
        be its generating function. By~\eqref{eq:s-meandric numbers from sum-indecomposable} and Lemma~\ref{lem:bound on gn}, we have $s_n \geq \hat{s}_n$ for every $n \geq 0$, where $\hat{s}_n$ is defined by
        $$
        \hat{s}_0=1,\qquad
        \hat{s}_n=\sum_{k=1}^n\hat{g}_k\hat{s}_{n-k}.
        $$
        It follows that
        $$
        \sum_{n \geq 0}\hat{s}_nz^n=\frac{1}{1-\hat{G}(z)}.
        $$
        Let $r\in\left(0;\,\frac{1}{2}\right)$ be the unique positive solution of $\hat{G}(r)=1$. Then $r$ is the radius of convergence of $(1-\hat{G}(z))^{-1}$, and hence\footnote{Here we may again use Fekete's lemma to show that the limit exists.}
        $$
        \lim_{n\to\infty}\sqrt[n]{\hat{s}_n}=\frac{1}{r}.
        $$
        Therefore,
        $$
        \lim_{n\to\infty}\sqrt[n]{s_n}
        \geq
        \lim_{n\to\infty}\sqrt[n]{\hat{s}_n}
        =
        \frac{1}{r}.
        $$
        Numerically, $r \approx 0.1980760601$, and thus $\mu \geq 5.04856568$.
    \end{proof}

    \begin{remark}
        Theorem~\ref{thm:estimate on s-meandric number} implies that meandric permutations form an exponentially small proportion of s-meandric permutations. Indeed, combining its lower bound with the upper bound of Albert and Paterson~\cite{AlbertPaterson05GrowthRate}, we obtain
        $$
        \limsup_{n\to\infty}\sqrt[n]{\frac{m_n}{s_n}}
        \leq\frac{3.5918}{5.0485}<0.712.
        $$
    \end{remark}

    \begin{figure}[h]
        \centering
        \includegraphics[width=\linewidth]{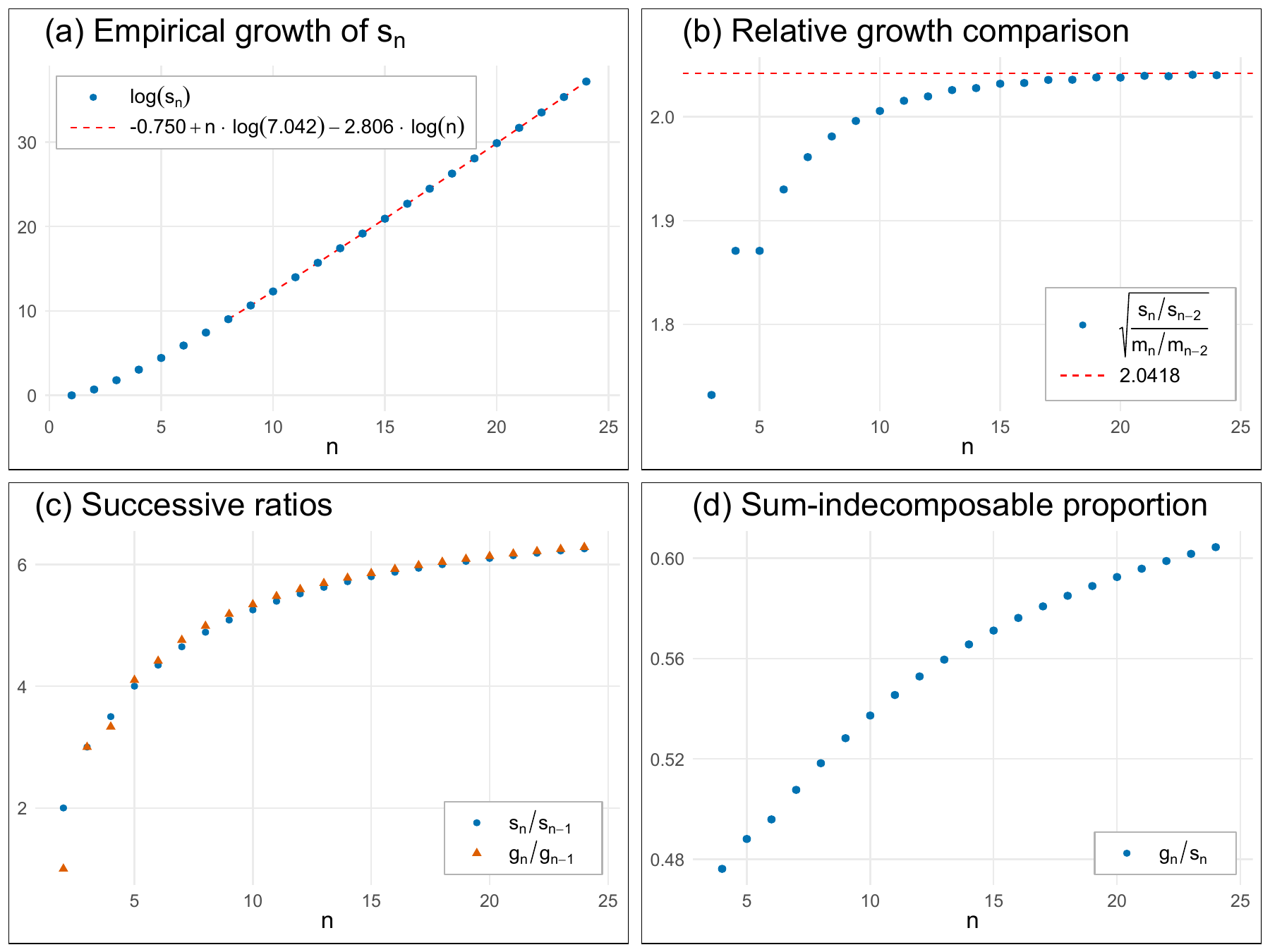}
        \caption{Observed growth of the numbers $s_n$ of s-meandric permutations.}
        \label{fig:growth s-n}
    \end{figure}
    
    \begin{remark}
        Let us give a few observations on the computed values.
        \begin{itemize}
            \item The numerical estimates suggest that $\mu \approx 7$, and, in fact, close to twice the observed exponential growth of meanders; see Fig.~\ref{fig:growth s-n} (a) and (b).
            \item Both computed sequences $s_n$ and $g_n$ are log-convex and have close successive ratios; see Fig.~\ref{fig:growth s-n} (c).
            \item The ratio $g_n/s_n$ increases from $0.4762$ at $n=4$ to $0.6044$ at $n=24$. This suggests that $g_n$ and $s_n$ have the same exponential growth and that a positive proportion of s-meandric permutations may be sum-indecomposable; see Fig.~\ref{fig:growth s-n} (d).
        \end{itemize}
    \end{remark}

    \begin{lemma}\label{lem:bound on gn}
    Let $g_n$ be the number of sum-indecomposable s-meandric permutations in $S_n$, and let $n \geq 2$. Then
    $$
    g_{n+1} \geq 2g_{n}.
    $$
    \end{lemma}

    \begin{proof}
        Let $\pi\in S_n$, where $n \geq 2$, and write $\pi_i=\pi(i)$, with $\pi_0=0$ and $\pi_{n+1}=n+1$. Fix a position $i \in \{1,\dots,n\}$ and put $x=\pi_i$.
        Consider two permutations in $S_{n+1}$:
        \begin{align*}
            \pi^{i,+} &= (\pi_1',\dots,\pi_{i-1}',x,x+1,\pi_{i+1}',\dots,\pi_n'),\\
            \pi^{i,-} &= (\pi_1',\dots,\pi_{i-1}',x+1,x,\pi_{i+1}',\dots,\pi_n'),
        \end{align*}
        where
        $$
        \pi_k'=
        \begin{cases}
            \pi_k, & \pi_k\leq x,\\
            \pi_k+1, & \pi_k>x.
        \end{cases}
        $$
        It is easy to see that if $\pi$ is sum-indecomposable, then the same is true for both $\pi^{i,+}$ and $\pi^{i,-}$.
    
        Moreover, exactly one of these permutations has the interlacement graph obtained from $G_\pi$ by adding one isolated vertex, while the other has the interlacement graph obtained from $G_\pi$ by adding one isolated vertex and one extra edge between the vertices corresponding to the arcs $A_{i-1}$ and $A_i$. This is clear from the geometric point of view: see Fig.~\ref{fig:pi-plus and pi-minus}, where the solid arc corresponds to $A_i$ and the dashed arcs show the possible positions of $A_{i-1}$ (we assume $\pi_i<\pi_{i+1}$, while the other possibility corresponds to just switching the labels). We call the first operation a \emph{safe pinch} and the second a \emph{risky pinch}.\footnote{In permutation-pattern terminology, the operations $\pi^{i,+}$ and $\pi^{i,-}$ are often called \emph{inflations} of the entry $\pi_i$ by $12$ and $21$, respectively. We use the term ``pinch'' to emphasize its geometric origin.}
    
        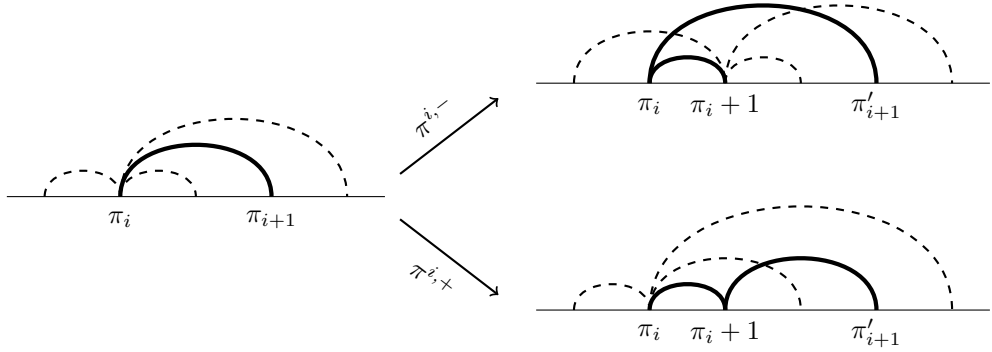
\begin{figure}[h]
            \centering
            \begin{tikzpicture}
                \newcommand{\deltax}{7}            
                \newcommand{\deltay}{1.5}
                \newcommand{\arcc}[4][]{\draw[#1] (#2,#4) to [out = 90, in = 90, distance=13*((#3)-(#2))] (#3,#4);}
                        
                \arcc[thick, dashed]{0}{1}{0}
                \arcc[thick, dashed]{1}{2}{0}
                \arcc[thick, dashed]{1}{4}{0}
                \arcc[ultra thick]{1}{3}{0}
                \draw (-0.5, 0) to (4.5, 0);
            
                \node at (1, 0-0.3) {$\pi_i$};
                \node at (3, 0-0.3) {$\pi_{i+1}$};
            
                \arcc[thick, dashed]{\deltax+0}{\deltax+2}{\deltay}
                \arcc[thick, dashed]{\deltax+2}{\deltax+3}{\deltay}
                \arcc[thick, dashed]{\deltax+2}{\deltax+5}{\deltay}
                \arcc[ultra thick]{\deltax+1}{\deltax+4}{\deltay}
                \arcc[ultra thick]{\deltax+1}{\deltax+2}{\deltay}
                \draw (\deltax-0.5, \deltay) to (\deltax+5.5, \deltay);
            
                \node at (\deltax+1, \deltay-0.3) {$\pi_i$};
                \node at (\deltax+4, \deltay-0.3) {$\pi_{i+1}'$};
                \node at (\deltax+2, \deltay-0.3) {$\pi_{i}+1$};
            
                \arcc[thick, dashed]{\deltax+0}{\deltax+1}{-\deltay}
                \arcc[thick, dashed]{\deltax+1}{\deltax+3}{-\deltay}
                \arcc[thick, dashed]{\deltax+1}{\deltax+5}{-\deltay}
                \arcc[ultra thick]{\deltax+2}{\deltax+4}{-\deltay}
                \arcc[ultra thick]{\deltax+1}{\deltax+2}{-\deltay}
                \draw (\deltax-0.5, -\deltay) to (\deltax+5.5, -\deltay);
            
                \node at (\deltax+1, -\deltay-0.3) {$\pi_i$};
                \node at (\deltax+4, -\deltay-0.3) {$\pi_{i+1}'$};
                \node at (\deltax+2, -\deltay-0.3) {$\pi_{i}+1$};
            
                \draw[thick, ->] (4.7,0.3) to node[midway,sloped,above=2pt] {$\pi^{i,-}$} (\deltax-1,\deltay-0.2);
                \draw[thick, ->] (4.7,-0.3) to node[midway,sloped,below=2pt] {$\pi^{i,+}$} (\deltax-1,-\deltay+0.2);
            \end{tikzpicture}
            \caption{The arcs of permutations $\pi^{i, -}$ and $\pi^{i, +}$.}
            \label{fig:pi-plus and pi-minus}
        \end{figure}
        
        If $\pi$ is s-meandric, every safe pinch produces an s-meandric permutation. Moreover, there is at least one risky pinch that produces an s-meandric permutation. Indeed, a risky pinch produces a permutation that is not s-meandric if and only if the arcs $A_{i-1}$ and $A_i$ lie in the same connected component and in the same bipartition class. If all risky pinches produced permutations that were not s-meandric, then all the arcs $A_0,\dots,A_n$ would lie in the same connected component and in the same bipartition class, which is impossible.
    
        We say that $i\in\{1,\dots,n-1\}$ is a \emph{cup}\footnote{In permutation-pattern terminology, such an adjacent pair is often called a \emph{bond}. We use the term ``cup'' as it is more standard in the theory of meanders.} in $\pi$ if $|\pi_i-\pi_{i+1}|=1$. A \emph{cup run} is a maximal consecutive sequence
        $$
        i,i+1,\dots,i+k,\qquad k\geq1,
        $$
        such that $i,i+1,\dots,i+k-1$ are cups. In this case we say that the \emph{length} of the cup run is $k+1$. The corresponding values form a monotone sequence, either
        $$
        \pi_i,\pi_i+1,\dots,\pi_i+k
        $$
        or
        $$
        \pi_i,\pi_i-1,\dots,\pi_i-k.
        $$
    
        For a given s-meandric permutation $\pi\in S_n$, safe pinches at two distinct positions produce the same permutation if and only if the two positions belong to the same cup run; in this case, the pinch simply increases the length of the corresponding cup run by one. Also note that if $i\in\{1,\dots,n-1\}$ is a cup in $\pi$, then the vertex $A_i$ of $G_\pi$ is isolated. Consequently, every risky pinch at a position belonging to a cup run produces an s-meandric permutation. Also, for a given permutation risky pinches at distinct positions produce distinct permutations.
    
        Thus, if $\pi$ is s-meandric and $r_1,\dots,r_t$ are the lengths of all its cup runs, there are exactly
        $$
        n-\sum_{j=1}^t(r_j-1)
        $$
        distinct permutations obtained by safe pinches and at least
        $$
        \max\left\{1,\sum_{j=1}^t r_j\right\}
        $$
        distinct s-meandric permutations obtained by risky pinches. Hence at least $n+1$ distinct s-meandric permutations are obtained from $\pi$ by pinching.
    
        Conversely, the number of distinct permutations in $S_n$ that lead to a given permutation $\pi'\in S_{n+1}$ by a pinch is equal to the number of cup runs in $\pi'$. The maximal possible number of cup runs in a permutation in $S_{n+1}$ is
        $
        \left\lfloor\frac{n+1}{2}\right\rfloor.
        $
        Therefore,
        $$
        (n+1)g_n \leq \left\lfloor\frac{n+1}{2}\right\rfloor g_{n+1}.
        $$
        Since $(n+1)/\lfloor (n+1)/2\rfloor\geq2$, it follows that $g_{n+1}\geq2g_{n}$.
    \end{proof} 
    
    \begin{remark}
        In fact, the proof of Lemma~\ref{lem:bound on gn} establishes the stronger bound for $n \geq 3$
        $$
        g_n \geq \left\lceil \frac{n}{\lfloor n/2 \rfloor}g_{n-1} \right\rceil.
        $$
        But this makes the calculations in the proof of the lower bound in Theorem~\ref{thm:estimate on s-meandric number} more complicated, while providing only a tiny increment.
    \end{remark}

\section{Numbers of distinct tangency polynomials}\label{sec:number-tangency-polynomials}
A tangency polynomial does not uniquely determine its block factors $E_{\lambda,\eta}(x)$. In this section, we describe all multiplicative identities between these factors and use them to obtain an asymptotic formula for the number of distinct tangency polynomials.

\subsection{Deficient and saturated factors}
We use the notation from Section~\ref{sec:block factorization}. For each integer $m\geq1$, set
\begin{align*}
    D_m(x):=& \frac{(1+x)^m - (-1)^m(1-x)^m}{2} = E_{m,\eta}(x), \quad \eta \equiv m-1 \pmod 2,\\
    \widehat D_m(x):=&\frac{(1+x)^m + (-1)^m(1-x)^m}{2} = E_{m,\eta}(x),\quad \eta\equiv m\pmod2,
\end{align*}
where $\eta\in\{0,1\}$ in each line. Thus, $\deg D_m=m-1$ and $\deg\widehat D_m=m$. In particular, $D_1(x)=1$ and $\widehat D_1(x)=x$.

\begin{definition}
    We call $D_m$ the \emph{deficient factor} of length $m$ and $\widehat D_m$ the \emph{saturated factor} of length $m$. A \emph{formal factor datum} consists of a nonnegative integer $b$ and finite multisets of deficient factors of lengths at least $2$ and saturated factors of lengths at least $1$. Its polynomial and its \emph{visible cost} are, respectively,
    $$
    (1+x)^b\prod_i D_{d_i}(x)\prod_j\widehat D_{s_j}(x)
    \qquad\text{and}\qquad
    b+\sum_i d_i+\sum_j s_j.
    $$
\end{definition}

The factors $D_1(x)=1$ correspond to positive loops and are omitted from a formal factor datum.

\begin{theorem} \label{thm:visible realization}
    Every formal factor datum of visible cost $n$ that contains a deficient factor is realized by an s-meandric permutation in $S_n$.
\end{theorem}

\begin{proof}
    We first realize one deficient factor together with at least one saturated factor. Let $k \geq 2$, let $\lambda_1,\ldots,\lambda_{k-1}\geq1$, and let $\lambda_{k}\geq2$. Put $N:=\sum_{i=1}^{k} \lambda_i$. Partition $\{1, 2, \ldots,N\}$ into consecutive intervals $I_1<\cdots<I_{k}$ with $|I_i|=\lambda_i$ for $i\in \{1,2,\dots,k\}$. 
    Denote by $\uparrow I$ and $\downarrow I$ the strings consisting of the elements of $I$ written in increasing and decreasing order, respectively. Define the permutation $\Omega(\lambda_1,\ldots,\lambda_k)\in S_N$ by listing the even-indexed intervals in increasing order and then the odd-indexed intervals in decreasing order, with the following one-line notation:
    $$
    \Omega(\lambda_1,\ldots,\lambda_{k})=
    \uparrow I_2\,\uparrow I_4\,\uparrow I_6\cdots \uparrow I_{2\lfloor k/2\rfloor}
    \downarrow I_{2\lceil k/2\rceil-1}\cdots
    \downarrow I_3\,\downarrow I_1.
    $$
    For example,
    $$
    \Omega(3,2,2,4) = (\underbrace{4,5}_{\uparrow I_2},\underbrace{8,9,10,11}_{\uparrow I_4},\underbrace{7,6}_{\downarrow I_3},\underbrace{3,2,1}_{\downarrow I_1}).
    $$
    We show that, under the stated conditions on the $\lambda_i$, these permutations are always s-meandric (see Fig.~\ref{fig:realization of omega} for an example).
    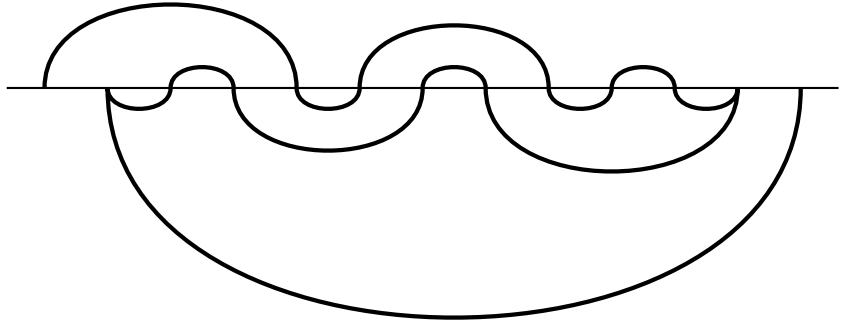
\begin{figure}[h]
        \centering
        \begin{tikzpicture}[scale = 10]
            \draw[thick] (-0.05, 0) to (1.05, 0);
            \draw[ultra thick] (-0, 0)
        	to[out = 90, in = 90, distance = 4.18879] (0.333333, 0)
        	to[out = -90, in = -90, distance = 1.0472] (0.416667, 0)
        	to[out = 90, in = 90, distance = 3.14159] (0.666667, 0)
        	to[out = -90, in = -90, distance = 1.0472] (0.75, 0)
        	to[out = 90, in = 90, distance = 1.0472] (0.833333, 0)
        	to[out = -90, in = -90, distance = 1.0472] (0.916667, 0)
        	to[out = -90, in = -90, distance = 4.18879] (0.583333, 0)
        	to[out = 90, in = 90, distance = 1.0472] (0.5, 0)
        	to[out = -90, in = -90, distance = 3.14159] (0.25, 0)
        	to[out = 90, in = 90, distance = 1.0472] (0.166667, 0)
        	to[out = -90, in = -90, distance = 1.0472] (0.0833333, 0)
        	to[out = -90, in = -90, distance = 11.5192] (1, 0);
        \end{tikzpicture}
        \caption{A singular meander realizing the permutation $\Omega(3,2,2,4) = (4,5,8,9,10,11,7,6,3,2,1)$.}
        \label{fig:realization of omega}
    \end{figure}
    
    For each $i\in\{1,\dots,k\}$, every arc of $\Omega(\lambda_1,\ldots,\lambda_k)$ with both endpoints in $I_i$ joins consecutive integers and is therefore isolated in the interlacement graph. Label the remaining arcs as follows:
    \begin{itemize}
        \item for $1\leq j\leq k-2$, let $L_j$ be the arc joining $\max I_j$ and $\min I_{j+2}$;
        \item let $L_*$ be the arc joining $\max I_{k-1}$ and $\max I_k$;
        \item let $L_0$ be the arc joining $0$ and $\min I_2$;
        \item let $L_{-1}$ be the arc joining $\min I_1$ and $N+1$.
    \end{itemize}
    
    For $-1\leq i,j\leq k-2$, the arcs $L_i$ and $L_j$ interlace if and only if $|i-j|=1$, and $L_*$ interlaces only $L_{k-2}$: here we use the condition $\lambda_k\geq2$; otherwise, $L_*$ and $L_{k-2}$ would share an endpoint instead of interlacing. Thus, we can define a proper coloring by

    $$
    \varepsilon(L_*):=k\pmod2, \qquad \text{and} \qquad
    \varepsilon(L_j):=j+1\pmod2
    \quad \text {for }-1\leq j\leq k-2.
    $$
    
    All other arcs correspond to isolated vertices and can be colored arbitrarily. Since $\varepsilon(L_0)=1$, this gives a realization of $\Omega(\lambda_1,\ldots,\lambda_k)$.

    The arcs $L_{-1},L_0,\dots,L_{k-2},L_*$ belong to a single connected component, while the $\lambda_i-1$ internal arcs of each interval $I_i$ are isolated. Hence the component spine is a bouquet of cycle blocks of lengths $\lambda_1,\ldots,\lambda_k$.

    The two arcs bounding $I_i$ have the same color for $i<k$ and opposite colors for $i=k$. A tangency coordinate is $1$ plus the colors of the two incident arcs in $\F_2$, so summing these coordinates over a block cancels all internal arc colors. Thus, the frustration parity of the block corresponding to $I_i$ is congruent to $\lambda_i$ modulo $2$ for $i<k$ and to $\lambda_k-1$ modulo $2$ for $i=k$. The block factorization therefore gives
    \begin{equation}\label{eq:tp-oscillating}
        T_{\Omega(\lambda_1,\ldots,\lambda_k)}(x)
        =
        D_{\lambda_k}(x)\prod_{i=1}^{k-1}\widehat D_{\lambda_i}(x).
    \end{equation}

    For the permutation $\pi=(m,m-1,\ldots,1)\in S_m$, where $m\geq2$, the component spine is a single cycle of length $m$ and frustration parity congruent to $m-1$ modulo $2$, so $T_\pi(x)=D_m(x)$. For $\operatorname{id}_m=(1,2,\dots,m)\in S_m$, where $m\geq1$, the component spine is a path with $m$ bridges, so $T_{\operatorname{id}_m}(x)=(1+x)^m$.

    Under a direct sum, the terminal vertex of the first component spine is identified with the initial vertex of the second, preserving all blocks and their signs. Thus, direct sums combine the corresponding formal factor data.

    To realize an arbitrary datum as in the theorem, choose one deficient factor and realize it together with all the saturated factors using $\Omega$ (or a decreasing permutation if there are no saturated factors). Realize the remaining deficient factors by decreasing permutations and, when $b>0$, the bridge factors by $\operatorname{id}_b$. Taking the direct sum gives a permutation of order equal to the visible cost and completes the proof.
\end{proof}

\subsection{Identities between the factors}
\begin{lemma}\label{lem:block factors identity}
    For every integer $m\geq1$,
    \begin{equation}\label{eq:block factors identity}
        D_{2m}(x)=2D_m(x)\widehat D_m(x).
    \end{equation}
    For every positive odd integer $r$ and every integer $a\geq0$, iteration gives
    \begin{equation}\label{eq:deficient factor factorization}
        D_{2^ar}(x)=2^aD_r(x)\prod_{j=0}^{a-1}\widehat D_{2^jr}(x).
    \end{equation}
\end{lemma}

\begin{proof}
    The first identity follows by direct substitution, and the second follows by iteration.
\end{proof}

\begin{theorem}\label{thm:polynomial normal form}
    Every polynomial of a formal factor datum, and hence every nonzero tangency polynomial, has a unique representation
    \begin{equation}\label{eq:polynomail normal form}
        T(x)=2^s(1+x)^b
        \prod_{\substack{r\geq3\\r\text{ odd}}}D_r(x)^{\alpha_r}
        \prod_{m\geq1}\widehat D_m(x)^{\beta_m},
    \end{equation}
    where all exponents are nonnegative integers and only finitely many $\alpha_r$ and $\beta_m$ are nonzero. Apart from $D_1(x)=1$, every multiplicative identity between the block factors follows from~\eqref{eq:deficient factor factorization} and commutativity.
\end{theorem}

\begin{proof}
    Existence follows from~\eqref{eq:deficient factor factorization}. To prove uniqueness, first note that $b$ is determined by the multiplicity of $-1$ as a root of $T(x)$ (see the proof of Corollary~\ref{cor:roots of tangency polynomials}). Let $u:=-(1+x)/(1-x)$. Then
    \begin{equation}\label{eq:polynomial factors after substitution}
        D_r(x)=2^{r-1}\frac{u^r-1}{(u-1)^r},
        \qquad
        \widehat D_m(x)=2^{m-1}\frac{u^m+1}{(u-1)^m}.
    \end{equation}
    Factoring the numerators into cyclotomic polynomials $\Phi_d(u)$ gives
    \begin{equation}\label{eq:cyclotomic factorization for odd}
        u^r-1=\prod_{d\mid r}\Phi_d(u)
    \end{equation}
    and
    \begin{equation}\label{eq:cyclotomic factorization for even}
        u^m+1=\frac{u^{2m}-1}{u^m-1}
        =\prod_{\substack{d\mid2m\\d\nmid m}}\Phi_d(u).
    \end{equation}

    Since $r$ is odd, every divisor of $r$ is also odd, and hence~\eqref{eq:cyclotomic factorization for odd} contains only odd-indexed cyclotomic factors (including $\Phi_r$ exactly once). On the other hand,~\eqref{eq:cyclotomic factorization for even} contains only even-indexed cyclotomic factors (including $\Phi_{2m}$ exactly once). The denominators are powers of $u-1=\Phi_1(u)$ and therefore do not affect the multiplicity of any $\Phi_d(u)$ with $d\geq2$.

    Suppose there are two representations
    $$
    2^s(1+x)^b
    \prod_{\substack{r\geq3\\r\text{ odd}}}D_r(x)^{\alpha_r}
    \prod_{m\geq1}\widehat D_m(x)^{\beta_m}
    =
    2^{s'}(1+x)^{b'}
    \prod_{\substack{r\geq3\\r\text{ odd}}}D_r(x)^{\alpha_r'}
    \prod_{m\geq1}\widehat D_m(x)^{\beta_m'}.
    $$
    Since $b=b'$, we may cancel $(1+x)^b$. If $\alpha_r\neq\alpha_r'$ for some $r$, choose the largest such $R$, which exists by finite support. Cancel the common factors $D_r(x)^{\alpha_r}$ for all $r>R$. No $D_r(x)$ with $r<R$ contributes $\Phi_R(u)$, since $R\nmid r$, and no $\widehat D_m(x)$ contributes it, since $R$ is odd. Comparing its multiplicity therefore gives $\alpha_R=\alpha_R'$, a contradiction. Thus all $\alpha_r$ agree.

    Similarly, if some $\beta_m\neq\beta_m'$, choose the largest such $M$ and cancel the common factors with $m>M$. No factor $\widehat D_m(x)$ with $m<M$ contributes $\Phi_{2M}(u)$, and no deficient factor of odd length contributes it. Comparing its multiplicity gives $\beta_M=\beta_M'$, again a contradiction. Thus all $\beta_m$ agree. Finally, canceling the remaining common factors yields $2^s=2^{s'}$, hence $s=s'$.

    Reducing both sides of any multiplicative identity to this unique normal form using~\eqref{eq:block factors identity} and $D_1(x)=1$ proves the last assertion.
\end{proof}

\subsection{The number of distinct tangency polynomials}
In this subsection, we use the normal form~\eqref{eq:polynomail normal form} to determine the asymptotic behavior of the number of distinct tangency polynomials. Let
$$
\mathcal T_n:=\{T_\pi(x)\mid\pi\in S_n,\ \pi \text{ is s-meandric}\},
\qquad
N_n:=|\mathcal T_n|.
$$
For $n\geq0$, consider also the set $\mathcal U_n$ of polynomials admitting a formal factor datum of visible cost at most $n$, including $1$, represented by the empty product. These polynomials need not be tangency polynomials of permutations in $S_n$.

\begin{lemma}\label{lem:tn and un comparison}
    Let $U_n=|\mathcal U_n|$. Then for $n\geq9$,
    \begin{equation}\label{eq:tn and un comparison}
        U_{n-9}\leq N_n\leq U_n.
    \end{equation}
\end{lemma}

\begin{proof}
We first describe a local operation on singular meanders. Suppose that $M$ is the unique realization of its permutation and that two consecutive intersection points along $\ell$ are touches from opposite half-planes. Perform the local operation shown in Fig.~\ref{fig:local braid action}, which we call a \emph{braid action}, and perform a plane isotopy to normalize the resulting configuration and obtain a singular meander.
The action adds four transverse intersections and preserves the number of touches. The two touches in the replacement are again consecutive along $\ell$, so the action can be repeated.

\begin{figure}[ht]
    \centering
    \begin{tikzpicture}[scale = 3.5,
        intersection label/.style={font=\scriptsize, inner sep=0.5pt}
    ]
        \draw[thick] (-0.05, 0) to (1.05, 0);
        \begin{scope}
            \clip (-0.05, -0.35) rectangle (1.05, 0.35);
            \draw[ultra thick] (1.50, 0)
                to[out = -90, in = -90, distance = 15] (0.25, 0)
                to[out = -90, in = -90, distance = 26.38938] (2.35, 0);
            \draw[ultra thick] (-0.60, 0)
                to[out = 90, in = 90, distance = 14] (0.70, 0)
                to[out = 90, in = 90, distance = 14] (1.30, 0);
        \end{scope}
        \node[intersection label, above=2pt] at (0.25, 0) {$p_k$};
        \node[intersection label, below=2pt] at (0.70, 0) {$p_{k+1}$};
        \node[right] at (1.05, 0) {$\ell$};

        \draw[thick, ->] (1.25, 0) to (1.74, 0);

        \begin{scope}[xshift = 1.95cm]
            \draw[thick] (-0.05, 0) to (1.05, 0);
            \begin{scope}
                \clip (-0.05, -0.35) rectangle (1.05, 0.35);
                \draw[ultra thick] (2.35, 0)
                    to[out = -90, in = -90, distance = 28.27433] (0.10, 0)
                    to[out = 90, in = 90, distance = 6.03186] (0.58, 0)
                    to[out = 90, in = 90, distance = 4.02124] (0.26, 0)
                    to[out = -90, in = -90, distance = 15.58230] (1.50, 0);
                \draw[ultra thick] (-0.60, 0)
                    to[out = 90, in = 90, distance = 16.83894] (0.74, 0)
                    to[out = -90, in = -90, distance = 4.02124] (0.42, 0)
                    to[out = -90, in = -90, distance = 6.03186] (0.90, 0)
                    to[out = 90, in = 90, distance = 20.10619] (2.50, 0);
            \end{scope}
            \node[intersection label, above left=2pt] at (0.10, 0) {$q_1$};
            \node[intersection label, below left=2pt] at (0.26, 0) {$q_2$};
            \node[intersection label, above=2pt] at (0.42, 0) {$q_3$};
            \node[intersection label, below=2pt] at (0.58, 0) {$q_4$};
            \node[intersection label, above right=2pt] at (0.74, 0) {$q_5$};
            \node[intersection label, below right=2pt] at (0.90, 0) {$q_6$};
            \node[right] at (1.05, 0) {$\ell$};
        \end{scope}
    \end{tikzpicture}
    \caption{A braid action at two consecutive touches from opposite sides of $\ell$. The points $q_1,q_2,q_5,q_6$ are transverse intersections, and $q_3,q_4$ are touches.}
    \label{fig:local braid action}
\end{figure}
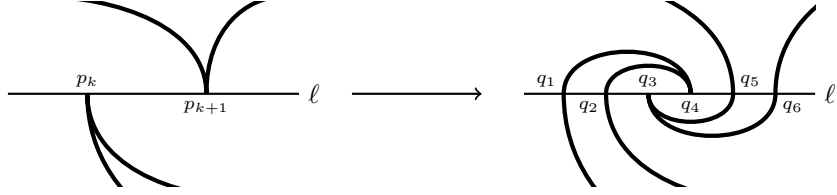

The resulting meander is again the unique realization of its permutation. Indeed, all interlacements between old arcs are preserved. The four new arcs form a $4$-cycle in the interlacement graph, and the arc $(q_1,q_4)$ interlaces the old arc incident to $q_2$. Since the old interlacement graph is connected, so is the new one. Theorem~\ref{thm:degree of realizability} therefore gives uniqueness.

Now consider the four singular meanders\footnote{These meanders are irreducible in the terminology of~\cite{Belousov26PrimeFactorization}.} in Fig.~\ref{fig:braid action examples}. Each has six touches and a connected interlacement graph, hence is the unique realization of its permutation. In each case, $p_1$ is a touch from below and $p_2$ is a touch from above. Repeated braid actions therefore show that $x^6\in\mathcal T_n$ for every $n\geq7$: the four initial orders $7,8,9,10$ cover all residue classes modulo $4$.

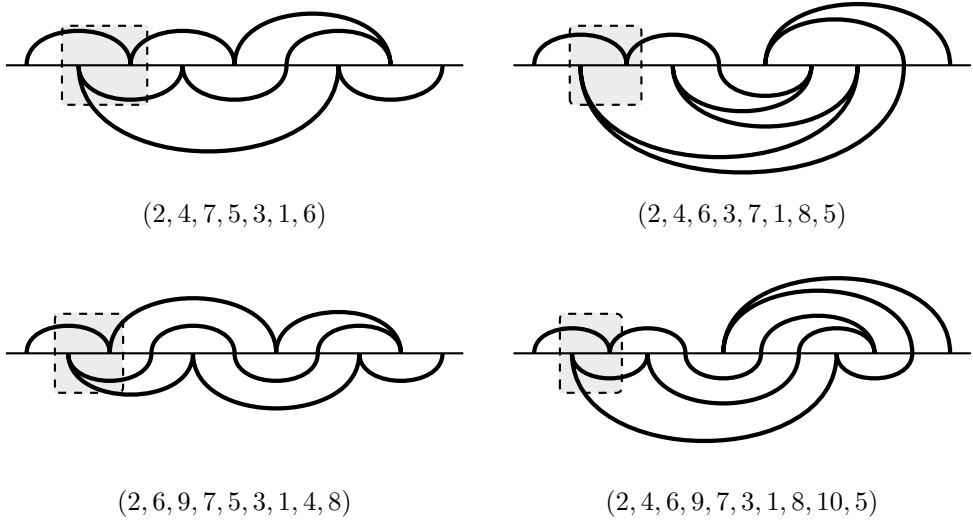
\begin{figure}[h]
    \centering
    \begin{tikzpicture}[scale = 5.5,
        intersection label/.style={font=\scriptsize, inner sep=0.5pt}
    ]
        \fill[gray!15] (0.085000, -0.095) rectangle (0.290000, 0.095);
        \draw[thick] (-0.05, 0) to (1.05, 0);
        \draw[ultra thick] (0, 0)
            to[out = 90, in = 90, distance = 3.14159] (0.250000, 0)
            to[out = 90, in = 90, distance = 3.14159] (0.500000, 0)
            to[out = 90, in = 90, distance = 4.71239] (0.875000, 0)
            to[out = 90, in = 90, distance = 3.14159] (0.625000, 0)
            to[out = -90, in = -90, distance = 3.14159] (0.375000, 0)
            to[out = -90, in = -90, distance = 3.14159] (0.125000, 0)
            to[out = -90, in = -90, distance = 7.85398] (0.750000, 0)
            to[out = -90, in = -90, distance = 3.14159] (1.000000, 0);
        \draw[thick, dashed, rounded corners=2pt] (0.085000, -0.095) rectangle (0.290000, 0.095);
        \node at (0.5, -0.36) {$(2,4,7,5,3,1,6)$};
    \end{tikzpicture}
    \hspace{0.4cm}
    \begin{tikzpicture}[scale = 5.5,
        intersection label/.style={font=\scriptsize, inner sep=0.5pt}
    ]
        \fill[gray!15] (0.085556, -0.095) rectangle (0.257778, 0.095);
        \draw[thick] (-0.05, 0) to (1.05, 0);
        \draw[ultra thick] (0, 0)
            to[out = 90, in = 90, distance = 2.79253] (0.222222, 0)
            to[out = 90, in = 90, distance = 2.79253] (0.444444, 0)
            to[out = -90, in = -90, distance = 2.79253] (0.666667, 0)
            to[out = -90, in = -90, distance = 4.18879] (0.333333, 0)
            to[out = -90, in = -90, distance = 5.58505] (0.777778, 0)
            to[out = -90, in = -90, distance = 8.37758] (0.111111, 0)
            to[out = -90, in = -90, distance = 9.77384] (0.888889, 0)
            to[out = 90, in = 90, distance = 4.18879] (0.555556, 0)
            to[out = 90, in = 90, distance = 5.58505] (1.000000, 0);
        \draw[thick, dashed, rounded corners=2pt] (0.085556, -0.095) rectangle (0.257778, 0.095);
        \node at (0.5, -0.36) {$(2,4,6,3,7,1,8,5)$};
    \end{tikzpicture}

    \par\medskip

    \begin{tikzpicture}[scale = 5.5,
        intersection label/.style={font=\scriptsize, inner sep=0.5pt}
    ]
        \fill[gray!15] (0.068000, -0.095) rectangle (0.232000, 0.095);
        \draw[thick] (-0.05, 0) to (1.05, 0);
        \draw[ultra thick] (0, 0)
            to[out = 90, in = 90, distance = 2.51327] (0.200000, 0)
            to[out = 90, in = 90, distance = 5.02655] (0.600000, 0)
            to[out = 90, in = 90, distance = 3.76991] (0.900000, 0)
            to[out = 90, in = 90, distance = 2.51327] (0.700000, 0)
            to[out = -90, in = -90, distance = 2.51327] (0.500000, 0)
            to[out = 90, in = 90, distance = 2.51327] (0.300000, 0)
            to[out = -90, in = -90, distance = 2.51327] (0.100000, 0)
            to[out = -90, in = -90, distance = 3.76991] (0.400000, 0)
            to[out = -90, in = -90, distance = 5.02655] (0.800000, 0)
            to[out = -90, in = -90, distance = 2.51327] (1.000000, 0);
        \draw[thick, dashed, rounded corners=2pt] (0.068000, -0.095) rectangle (0.232000, 0.095);
        \node at (0.5, -0.36) {$(2,6,9,7,5,3,1,4,8)$};
    \end{tikzpicture}
    \hspace{0.4cm}
    \begin{tikzpicture}[scale = 5.5,
        intersection label/.style={font=\scriptsize, inner sep=0.5pt}
    ]
        \fill[gray!15] (0.061818, -0.095) rectangle (0.210909, 0.095);
        \draw[thick] (-0.05, 0) to (1.05, 0);
        \draw[ultra thick] (0, 0)
            to[out = 90, in = 90, distance = 2.28479] (0.181818, 0)
            to[out = 90, in = 90, distance = 2.28479] (0.363636, 0)
            to[out = -90, in = -90, distance = 2.28479] (0.545455, 0)
            to[out = 90, in = 90, distance = 3.42719] (0.818182, 0)
            to[out = 90, in = 90, distance = 2.28479] (0.636364, 0)
            to[out = -90, in = -90, distance = 4.56959] (0.272727, 0)
            to[out = -90, in = -90, distance = 2.28479] (0.090909, 0)
            to[out = -90, in = -90, distance = 7.99678] (0.727273, 0)
            to[out = -90, in = -90, distance = 2.28479] (0.909091, 0)
            to[out = 90, in = 90, distance = 5.71199] (0.454545, 0)
            to[out = 90, in = 90, distance = 6.85438] (1.000000, 0);
        \draw[thick, dashed, rounded corners=2pt] (0.061818, -0.095) rectangle (0.210909, 0.095);
        \node at (0.5, -0.36) {$(2,4,6,9,7,3,1,8,10,5)$};
    \end{tikzpicture}
    \caption{The four initial singular meanders. The shaded areas indicate
    where the braid action is applied.}
    \label{fig:braid action examples}
\end{figure}

Let $F\in\mathcal U_{n-9}$, and choose a formal factor datum for $F$ of visible cost $v\leq n-9$. Appending $D_2(x)=2x$ and applying Theorem~\ref{thm:visible realization} realizes $2xF$ at order $v+2$. Taking the direct sum with a permutation of order $n-v-2\geq7$ and tangency polynomial $x^6$ realizes $2x^7F$ at order $n$, by Theorem~\ref{thm:poly sum of permutations}. Multiplication by $2x^7$ is injective, which proves the lower bound.

The upper bound follows from Theorem~\ref{thm:block factorization}: omitting the positive-loop factors gives a formal factor datum of visible cost at most $n$ for every polynomial in $\mathcal T_n$.
\end{proof}

Let $P(x)\in\mathcal U_n$. Consider its normal form~\eqref{eq:polynomail normal form} and define its \emph{weight} $W(P)$ by
$$
W(P):=b+\sum_{\substack{r\geq3\\r\text{ odd}}}r\alpha_r+\sum_{m\geq1}m\beta_m.
$$
Note that the weight need not equal the degree of $P$. Omitting the scalar exponent $s$, we associate with $P$ the exponent vector $(b,(\alpha_r),(\beta_m))$. As $n$ varies, the generating function for the distinct such vectors, counted by weight, is
$$
H(z)=\frac{1}{1-z}
\prod_{\substack{r\geq3\\r\text{ odd}}}\frac{1}{1-z^r}
\prod_{m\geq1}\frac{1}{1-z^m}.
$$
We can also rewrite it as
$$
H(z)=\prod_{m\geq1}\frac{1+z^m}{1-z^m},
$$
which is the generating function for overpartitions (see, for example,~\cite{Barqueroetall2023Overpartitions}).

\begin{theorem}\label{thm:eq:formal polynomial generating functions}
    \begin{equation}\label{eq:formal polynomial generating functions}
        \sum_{n\geq0}U_nz^n
        =
        \frac{H(z)}{1-z}
        \left(1+\sum_{m\geq1}\frac{z^{2m}}{1-z^{2m}}\right).
    \end{equation}
\end{theorem}

\begin{proof}
Fix a finitely supported vector $(b,(\alpha_r),(\beta_m))$ of nonnegative integers with
$$
W:=b+\sum_{\substack{r\geq3\\r\text{ odd}}}r\alpha_r+\sum_{m\geq1}m\beta_m\leq n,
$$
and put $\alpha_1:=n-W$, which records the remaining cost budget. We claim that the possible scalar exponents $s$ for polynomials in $\mathcal U_n$ with this vector are precisely
\begin{equation}\label{eq:scalar range}
    0\leq s\leq
    \sum_{\substack{r\geq1\\r\text{ odd}}}\sum_{j\geq0}
    \min(\alpha_r,\beta_r,\beta_{2r},\ldots,\beta_{2^jr}).
\end{equation}
Indeed, fix a positive odd integer $r$, and let $t_j$ count the deficient factors $D_{2^ar}$ with $a>j$ in a formal factor datum before reduction. By~\eqref{eq:block factors identity}, their contribution to $s$ is $\sum_jt_j$, and the constraints are
$$
\alpha_r\geq t_0\geq t_1\geq\cdots\geq0,
\qquad
t_j\leq\beta_{2^jr}.
$$
For $r=1$, the first inequality accounts for the extra unit of cost for each deficient factor whose length is a power of two.

These conditions are also sufficient: for each $j\geq0$, take $t_j-t_{j+1}$ deficient factors of length $2^{j+1}r$, retain $\beta_{2^jr}-t_j$ saturated factors of length $2^jr$, and, for $r\geq3$, retain $\alpha_r-t_0$ deficient factors of length $r$. The resulting total visible cost is $W$ plus the value of $t_0$ for $r=1$, and is therefore at most $n$. The largest admissible $t_j$ is the corresponding successive minimum in~\eqref{eq:scalar range}. Repeatedly decreasing the last nonzero entry shows that every smaller sum is possible. Summing over positive odd $r$ proves the claim.

By uniqueness of the normal form, the number of polynomials associated with the fixed exponent vector is one plus the upper bound in~\eqref{eq:scalar range}. The exponent vectors together with $\alpha_1$, counted by $W+\alpha_1$, have generating function $H(z)/(1-z)$. To sum a minimum in~\eqref{eq:scalar range}, write it as the number of positive integers $h$ for which $\alpha_r,\beta_r,\ldots,\beta_{2^jr}\geq h$. This restriction adds weight
$$
h\left(r+\sum_{i=0}^j2^ir\right)=h2^{j+1}r.
$$
Its contribution is therefore
$$
\frac{H(z)}{1-z}\,
\frac{z^{2^{j+1}r}}{1-z^{2^{j+1}r}}.
$$
Every positive even integer occurs exactly once as $2^{j+1}r$ with $r$ positive and odd. Adding the contribution of $s=0$ completes the proof.
\end{proof}

\begin{theorem}\label{thm:tangecy polinomial asymptotic}
    The number $N_n$ of distinct nonzero tangency polynomials satisfies
    \begin{equation}\label{eq:tangecy polinomial asymptotic}
        N_n=\frac{e^{\pi\sqrt n}}{8\pi^2}
        \left(\log n+2\gamma-2\log\pi
        +O\!\left(\frac{\log n}{\sqrt n}\right)\right)
    \end{equation}
    as $n\to\infty$, where $\gamma$ is Euler's constant. In particular,
    $N_n\sim(\log n)e^{\pi\sqrt n}/(8\pi^2)$.
\end{theorem}

\begin{proof}
    Put $h_n=[z^n](H(z)/(1-z))$. Since $H(z)$ counts overpartitions, the classical estimate for overpartitions (see~\cite[Proposition~5.1]{Barqueroetall2023Overpartitions}) gives
    $$
    [z^n]H(z)=\frac{e^{\pi\sqrt n}}{8n}
    \left(1+O(n^{-1/2})\right).
    $$
    Summation gives
    \begin{equation}\label{eq:overpartitions summing}
        h_n=\frac{e^{\pi\sqrt n}}{4\pi\sqrt n}
        \bigl(1+O(n^{-1/2})\bigr).
    \end{equation}
    We also need the estimate
    \begin{equation}\label{eq:lambert}
        L\left(e^{-t}\right):=
        \sum_{m\geq1}\frac{1}{e^{2mt}-1}
        =
        \frac{\log(1/(2t))+\gamma}{2t}+O(1),
        \qquad t\to0^+.
    \end{equation}
    For completeness, the function
    $f(y)=(e^y-1)^{-1}-e^{-y}/y$
    extends smoothly to $0$, has an integrable derivative on $[0;\infty)$, and satisfies
    $\int_0^\infty f(y)\,dy=\gamma$.
    Comparing its Riemann sum with its integral and using
    $\sum_{m\geq1}e^{-mu}/m=-\log(1-e^{-u})$ gives
    $$
    \sum_{m\geq1}\frac{1}{e^{mu}-1}
    =
    \frac{\log(1/u)+\gamma}{u}+O(1),
    \qquad u\to0^+.
    $$
    Taking $u=2t$ proves~\eqref{eq:lambert}.

    Write $L(z)=\sum_{k\geq1}a_kz^k$ and set $t_n=\pi/(2\sqrt n)$. Uniformly for $0\leq k\leq n^{5/8}$, \eqref{eq:overpartitions summing} gives
    $$
    \frac{h_{n-k}}{h_n}
    =
    e^{-t_nk}
    \left(1+O\!\left(n^{-1/2}+\frac{k}{n}
    +\frac{k^2}{n^{3/2}}\right)\right).
    $$
    Since $0\leq a_k\leq k$ and $h_j$ is nondecreasing, the preceding estimate at $k=\lfloor n^{5/8}\rfloor$ shows that the convolution terms with $k>n^{5/8}$ contribute
    $O(h_nn^2e^{-cn^{1/8}})$ for some constant $c>0$.
    The corresponding tail of $h_nL(e^{-t_n})$ satisfies the same bound.

    Moreover, positivity and~\eqref{eq:lambert} imply
    $$
    \sum_{k\geq1}k^ja_ke^{-tk}
    =
    O\!\left(t^{-j-1}\log(1/t)\right),
    \qquad j=0,1,2,
    \qquad t\to0^+.
    $$
    For $j>0$, this follows from
    $k^je^{-tk}\leq C_jt^{-j}e^{-tk/2}$.
    Consequently, the convolution in~\eqref{eq:formal polynomial generating functions} yields
    $$
    \begin{aligned}
        U_n
        &=h_n+\sum_{k=1}^na_kh_{n-k}\\
        &=h_n\bigl(1+L(e^{-t_n})\bigr)
        +O\!\left(\frac{e^{\pi\sqrt n}\log n}{\sqrt n}\right)\\
        &=\frac{e^{\pi\sqrt n}}{8\pi^2}
        \left(\log n+2\gamma-2\log\pi
        +O\!\left(\frac{\log n}{\sqrt n}\right)\right).
    \end{aligned}
    $$

    For every fixed integer $k\geq0$, replacing $n$ by $n-k$ changes the main expression by
    $O(e^{\pi\sqrt n}\log n/\sqrt n)$.
    Thus, $U_n$ and $U_{n-9}$ have the same expansion with the stated error term, and the comparison in Lemma~\ref{lem:tn and un comparison} proves~\eqref{eq:tangecy polinomial asymptotic}.
\end{proof}

\subsection{Tangency polynomials of meandric permutations}
The same realization constructions give a sharp asymptotic formula for the number of distinct tangency polynomials of meandric permutations. Let $p(j)$ be the integer-partition function, with $p(0)=1$, and define
$$
H_n=\sum_{j=0}^n p(j).
$$

\begin{theorem}\label{thm:ordinary tangecy polinomial asymptotic}
    Let $O_n$ be the number of distinct tangency polynomials of meandric permutations in $S_n$. For $n\geq16$,
    \begin{equation}\label{eq:ordinary tangecy polinomial comparison}
        H_{n-16}\leq O_n\leq H_n.
    \end{equation}
    Consequently,
    \begin{equation}\label{eq:ordinary tangecy polinomial asymptotic}
        O_n\sim\frac{\exp\!\left(\pi\sqrt{2n/3}\right)}{2\pi\sqrt{2n}}
    \end{equation}
    as $n\to\infty$.
\end{theorem}

\begin{proof}
    By Corollary~\ref{cor:meandric permutation balanced spine}, for a meandric permutation we have
    $$
    T_\pi(x)=(1+x)^b\prod_{\lambda\geq2}E_{\lambda,0}(x)^{c_\lambda},
    \qquad b+\sum_{\lambda\geq2}\lambda c_\lambda\leq n.
    $$
    These products have unique factor data. The multiplicity of $-1$ as a root determines $b$. Among the remaining factors, if any, the largest length $L$ is determined by the root of smallest positive imaginary part, $i\tan(\pi/(2L))$. This root is simple in $E_{L,0}$ and does not occur in any shorter factor, so its multiplicity determines $c_L$. Dividing by $E_{L,0}^{c_L}$ and repeating recovers all lengths. Thus, products of visible cost $j$ are counted by $p(j)$, with parts of size one representing bridges. Positive loops contribute the factor one and account for the unused cost. This proves the upper bound.

    A safe pinch, defined in the proof of Lemma~\ref{lem:bound on gn}, adds an isolated vertex to the interlacement graph. At a forced touch, the two old incident arcs have the same color in every realization. The new isolated arc can independently take either color, giving two touches or two transverse intersections. Thus, a safe pinch replaces a factor $x$ by $1+x^2$ in the tangency polynomial, increases the order by one, and preserves every other forced touch. The proof of Lemma~\ref{lem:tn and un comparison} realizes $x^6$ at every order at least $7$. Applying safe pinches at the six original touches of such a realization therefore realizes
    $$
    Q(x)=(1+x^2)^6
    $$
    as a tangency polynomial in $\mathcal T_n$ for every $n\geq13$. These permutations are meandric since $Q(0)=1$.

    Let $F(x)$ be any product of $1+x$ and the factors $E_{\lambda,0}(x)$ for $\lambda\geq2$, with visible cost $v\leq n-16$. Adjoin the deficient factor $D_3(x)=1+3x^2$. Theorem~\ref{thm:visible realization} realizes $(1+3x^2)F(x)$ at order $v+3$. Taking its direct sum with a permutation having tangency polynomial $Q(x)$ at order $n-v-3\geq13$ gives the polynomial $Q(x)(1+3x^2)F(x)$, by Theorem~\ref{thm:poly sum of permutations}. Its constant term is one, so the resulting permutation is meandric. Multiplication by the fixed nonzero polynomial $Q(x)(1+3x^2)$ is injective, proving the lower bound.

    Finally, the Hardy--Ramanujan formula~\cite[formula~(1.41)]{HardyRamanujan1918AsymptoticFormulae} gives
    $$
    p(n)\sim\frac{\exp\!\left(\pi\sqrt{2n/3}\right)}{4\sqrt{3}\,n}.
    $$
    Put $B_n=\exp\!\left(\pi\sqrt{2n/3}\right)/(2\pi\sqrt{2n})$. Then $B_n-B_{n-1}\sim p(n)$, so the Stolz--Ces\`aro theorem gives $H_n\sim B_n$. Since $H_{n-16}/H_n\to1$, comparison~\eqref{eq:ordinary tangecy polinomial comparison} proves~\eqref{eq:ordinary tangecy polinomial asymptotic}.
\end{proof}

\begin{remark}
    The values $N_n$ and $O_n$ were computed for $n \leq 24$ by recursively generating the realizable multisets of bridge and signed cycle factors at each order. The results are presented in Appendix~\ref{appendix}. For details on the enumeration algorithm see~\cite{Bcode}.
\end{remark}
    
\bibliography{s_meandric_biblio}
\bibliographystyle{alpha}
\newpage

\appendix
\section{Enumeration data and related meandric numbers}\label{appendix}
\begin{table}[h!]
  \centering
  \caption{Enumeration data through $n=24$. Here $m_n$ counts meandric permutations, $g_n$ counts sum-indecomposable s-meandric permutations, $s_n$ counts all $s$-meandric permutations, and $M_n^{\mathrm{sing}}$ counts singular meanders with $n$ interior intersections. The numbers $N_n$ and $O_n$ count distinct tangency polynomials of $s$-meandric and meandric permutations, respectively.}
  \label{tab:meandric-enumeration}
  \small
  \begin{tabular}{r|rrrr|rr}
    \hline
    $n$ & $m_n$ & $g_n$ & $s_n$ & $M_n^{\mathrm{sing}}$ & $N_n$ & $O_n$ \\
    \hline
    1 & 1 & 1 & 1 & 2 & 1 & 1 \\
    2 & 1 & 1 & 2 & 6 & 2 & 1 \\
    3 & 2 & 3 & 6 & 24 & 4 & 2 \\
    4 & 3 & 10 & 21 & 112 & 9 & 2 \\
    5 & 8 & 41 & 84 & 576 & 21 & 4 \\
    6 & 14 & 181 & 365 & 3180 & 46 & 5 \\
    7 & 42 & 861 & 1696 & 18540 & 96 & 10 \\
    8 & 81 & 4295 & 8287 & 112840 & 191 & 13 \\
    9 & 262 & 22264 & 42146 & 711016 & 362 & 25 \\
    10 & 538 & 118947 & 221383 & 4610036 & 654 & 34 \\
    11 & 1828 & 651465 & 1194310 & 30614932 & 1144 & 60 \\
    12 & 3926 & 3642759 & 6588823 & 207498060 & 1936 & 85 \\
    13 & 13820 & 20730797 & 37048242 & 1431260984 & 3184 & 139 \\
    14 & 30694 & 119779547 & 211762404 & 10024424284 & 5107 & 197 \\
    15 & 110954 & 701259531 & 1227800670 & 71158769844 & 8020 & 306 \\
    16 & 252939 & 4153476304 & 7208589543 & 511160782120 & 12351 & 427 \\
    17 & 933458 & 24854788513 & 42794919790 & 3710968813976 & 18699 & 630 \\
    18 & 2172830 & 150105419419 & 256585250332 & 27198487483336 & 27875 & 866 \\
    19 & 8152860 & 914044681124 & 1552128503360 & 201060296071624 & 40990 & 1232 \\
    20 & 19304190 & 5607630602158 & 9464640534050 & 1497907249045496 & 59529 & 1663 \\
    21 & 73424650 & 34636704299035 & 58135013155354 & 11238798692083888 & 85488 & 2296 \\
    22 & 176343390 & 215269557208478 & 359458448025994 & \textemdash & 121506 & 3045 \\
    23 & 678390116 & 1345531555533771 & 2236108374178094 & \textemdash & 171098 & 4101 \\
    24 & 1649008456 & 8454227740602060 & 13988018295583591 & \textemdash & 238865 & 5357 \\
    \hline
  \end{tabular}
\end{table}
\end{document}